\documentclass[11pt]{amsart}

\usepackage{amsmath}
\usepackage{amssymb}
\usepackage{graphicx}

\usepackage[top=2.5cm,bottom=2.5cm,left=3.2cm,right=3.2cm]{geometry}
\newtheorem{theorem}{Theorem}[section]

\newtheorem{lemma}[theorem]{Lemma}
\newtheorem{proposition}[theorem]{Proposition}
\theoremstyle{definition}
\newtheorem{definition}[theorem]{Definition}
\newtheorem{example}[theorem]{Example}
\newtheorem{remark}[theorem]{Remark}

\numberwithin{equation}{section}

\title[Backward-forward dynamical systems for monotone inclusion problems]{Variable metric backward-forward dynamical systems for monotone inclusion problems}

\author[P. Gautam]{Pankaj Gautam}

\address[P. Gautam]{Department of Mathematical Sciences, Indian Institute of Technology(BHU), Varanasi 221005, India}
\email{{\tt pgautam908@gmail.com}}

\author[D. R. Sahu]{D. R. Sahu}
\address[D. R. Sahu]{Department of Mathematics, Banaras Hindu University, Varanasi 221005, India}
\email{\tt drsahudr@gmail.com}

\author[J. C. Yao]{J. C. Yao}
\address[J. C. Yao]{Center for General Education,China Medical University,Taichung 40402,
	Taiwan }
\address[J. C. Yao]{Department of Applied Mathematics, National Sun Yat-sen University, Kaohsiung 804, Taiwan}
\email{\tt yaojc@mail.cmu.edu.tw}

\keywords{Dynamical systems; Monotone inclusion; Backward-forward algorithm; variable metric; convex optimization problems}

\begin{document}

\begin{abstract}
This paper investigates first-order variable metric backward forward dynamical systems associated with monotone inclusion and convex minimization problems in real Hilbert space. The operators are chosen so that the backward-forward dynamical system is closely related to the forward-backward dynamical system and has the same computational complexity. We show existence, uniqueness, and weak asymptotic convergence of the generated trajectories and strong convergence if one of the operators is uniformly monotone. We also establish that an equilibrium point of the trajectory is globally exponentially stable and monotone attractor. As a particular case, we explore similar perspectives of the trajectories generated by a dynamical system related to the minimization of the sum of a nonsmooth convex and a smooth convex function. Numerical examples are given to illustrate the convergence of trajectories.
\end{abstract}

\maketitle


\section{Introduction}
The monotone inclusion problem is
\begin{align}\label{12}
	\text{to~find}~u\in H\textrm{ such that }0\in (A+B)u,
\end{align}
where $H$ is a real Hilbert space,    $A: H\to 2^H$ is maximally monotone operator and $B:H\to H$ is $\beta$-cocoercive operator for $\beta>0$.
This structure is quite central due to the large range of problems involved in fields such as signal and image processing, partial differential equations, mechanics, convex optimization, statics, and game theory \cite{briceno2013monotone,bauschke2011convex,combettes2011proximal,glowinski1989augmented,tseng1991applications,mercier1979lectures}.

The first order dynamical system, which is linked with forward-backward algorithm to solve the problem (\ref{12}) was studied by Bot et al. \cite{boct2017dynamical} as follows:
\begin{eqnarray}\label{11}{
		\left\{
		\begin{array}{lc@{}c@{}r}
			\dot{u}(t)=\lambda (t)[J_{\gamma A}(I-\gamma B) u(t)-u(t)]\\[6pt]
			u(0)= u_0,
		\end{array}\right.
}\end{eqnarray}where $u_0\in H$, $A: H\to 2^H$ is maximal monotone operator, $B:H\to H$ is $\beta$-cocercive operator for $\beta >0$, $J_{\gamma A}$ is resolvent of operator $A$ for $\gamma>0$ and $\lambda:[0,\infty)\to [0,\infty)$ is a Lebesgue measurable function. They studied the convergence of trajectories when $\gamma\in (0,2\beta)$. Also, Bot et al. \cite{boct2018convergence} studied that the trajectory generated by dynamical system (\ref{11}) strongly converges with exponential rate to solution of problem (\ref{12}), when $A: H\to 2^H$ is maximal monotone and $B: H\to H$ is monotone, $\frac{1}{\beta}$-Lipschitz for $\beta>0$ such that sum of both the operators is $\rho$-strongly monotone for $\rho>0$.

To determine the zero of $\partial \Phi+ B$, Abbas et al. \cite{abbas2015dynamical} studied the dynamical system of the form:
\begin{eqnarray}{
		\left\{
		\begin{array}{lc@{}c@{}r}
			\dot{u}(t)+u(t)=\operatorname{prox_{\gamma \Phi}}(I-\gamma B) u(t)\nonumber\\[6pt]
			u(0)= u_0,
		\end{array}\right.
}\end{eqnarray}
where $\Phi: H \to \mathbb{R}_\infty$ is a proper, lower semicontinuous and convex function, $\partial \Phi$ is  subdifferential of $\Phi$, $B: H\to H$ is a cocoercive operator and $\operatorname{prox_{\gamma \Phi}}$ denotes the proximal point operator of $\gamma\Phi$.

Bot et al. \cite{boct2018convergence} also studied the dynamical system which is linked with the minimization of sum of smooth and nonsmooth function, which is as follows:
\begin{eqnarray}{
		\left\{
		\begin{array}{lc@{}c@{}r}
			\dot{u}(t)=\lambda(t)[\operatorname{prox_{\gamma f}}(I-\gamma \nabla g) u(t)\nonumber-u(t)]\\[6pt]
			u(0)= u_0,
		\end{array}\right.
}\end{eqnarray}
where $f:H\to \mathbb{R}_{\infty}$ is proper, lower semicontinuous, convex function and $g:H\to \mathbb{R}$ is convex function , $1/\beta$-Lipschitz continuous gradient for $\beta>0$ and Fr\text{$\acute{e}$}chet differentiable.

To minimize a smooth  convex function $f:H\to \mathbb{R}$  over the nonempty, closed, convex set $C\subseteq H$, Antipin \cite{antipin1994minimization} and Bolte \cite{bolte2003continuous} discussed the convergence of the trajectories governed by
\begin{eqnarray}\label{d12}
	{
		\left\{
		\begin{array}{lc@{}c@{}r}
			\dot{u}(t)+u(t)=P_C(I-\gamma \nabla f) u(t)\\[6pt]
			u(0)= u_0,
		\end{array}\right.
}\end{eqnarray}
where $\gamma>0$  and $P_C$ is the projection operator on the set $C$. Antipin \cite{antipin1994minimization} also obtained the exponential convergence rate of the trajectory for dynamical system (\ref{d12}).

Bot et al. \cite{bot2016second} studied existence, uniqueness, weak and strong convergence of the trajectories achieved by second-order dynamical systems linked with the problem to find zeros of the sum of two operators, in which one is a maximally monotone operator, and another one is cocoercive. Some implicit type dynamical systems have been already studied in the literature ( see \cite{boct2016approaching,boct2017second,abbas2014newton,abbas2015dynamical,attouch2011continuous,attouch2015dynamic}).

The backward-forward algorithm was studied by Attouch et al. \cite{attouch2018backward} to solve the monotone inclusion problem (\ref{12}). Operators are chosen so that they are closely associated with a forward-backward algorithm to solve the problem (\ref{12}). The forward-backward algorithms with a symmetric positive definite operator $M$, called a variable metric, were studied by \cite{chouzenoux2014variable,combettes2014variable,lorenz2015inertial}. Raguet et al. \cite{raguet2015preconditioning} studied generalized variable metric forward-backward algorithm by taking the operator $M$ strongly positive.

In this manuscript, we investigate the first order dynamical system, which is associated with the variable metric backward-forward method to solve structured monotone inclusion problem of the form:
\begin{align}
	\text{find}~ u\in H: 0\in (A+B)u,\nonumber
\end{align}
where $A: H\to 2^H$ is maximal $(\gamma-\alpha)$-cohypomonotone for $\gamma\in \mathbb{R}, \alpha >0$, $B:H\to H$ is a $\beta$-cocoercive for $\beta>0$ and $H$ is a real Hilbert space.
We study  first-order variable metric backward-forward dynamical system of the form:
\begin{eqnarray}\label{meq}{
		\left\{
		\begin{array}{lc@{}c@{}r}
			\dot{u}(t)=\lambda (t)[(I-\gamma M^{-1}B)J_{\gamma A}^M u(t)-u(t)]\\[6pt]
			u(0)= u_0,
		\end{array}\right.
}\end{eqnarray}
where $\gamma\neq 0$, $u_0\in H$, $\lambda: [0,\infty)\to [0,\infty)$ is a Lebesgue measurable function and $J_{\gamma A}^M: H\to 2^H$ is an operator defined by $J_{\gamma A}^M:= (I+\gamma M^{-1}A)^{-1}$ and $M:H\to H$ is a strongly positive operator. It is shown that the equilibrium point is exponentially stable and monotone attractor, whenever $B_{-\gamma}$ is $\rho$-strongly monotone for $\rho>0$.  

We study the convergence behaviour of the trajectories generated by forward-backward dynamical system in variable metric setting:
\begin{eqnarray}\label{meq2}{
		\left\{
		\begin{array}{lc@{}c@{}r}
			\dot{x}(t)=\lambda (t)[J_{\gamma A}^M (I-\gamma M^{-1}B)x(t)-x(t)]\\[6pt]
			x(0)= x_0,
		\end{array}\right.
}\end{eqnarray}
where $x_0\in H$, $\lambda: [0,\infty)\to [0,\infty)$ is a Lebesgue measurable function and operators $A$, $B$ and $M$ satisfy the same conditions as in dynamical system (\ref{meq}).

We also examine the first order dynamical system generated by optimization problem of the form:
\begin{align}\label{eq1}
	\min_{u\in H}f(u)+g(u),
\end{align}
where $f:H\to \mathbb{R}\cup\{\infty\}$ is proper, convex and lower semicontinuous function and $g: H\to \mathbb{R}$ is differentiable such that its gradient $\nabla g$ is $\beta$-cocercive for $\beta>0$.

The remaining parts of this paper are organized as follows: some lemmas and definitions required for proving the main results are presented in Section $\ref{sc1}$. Existence, uniqueness, and convergence of the trajectories generated by the first-order backward-forward dynamical system (\ref{meq}) and forward-backward dynamical system (\ref{meq2}) in the variable metric environment are studied in Section \ref{sc2}. In this section, we also study the convergence behavior of the dynamical system's trajectories, which is associated with minimizing the sum of a smooth and nonsmooth function. Finally, Section \ref{sc3} is devoted to numerical experiments to illustrate the convergence of the trajectories of the dynamical system (\ref{meq}).  
\section{Preliminaries}\label{sc1}
Throughout this paper $H$ denotes a real Hilbert space with inner product $\langle\cdot, \cdot\rangle$, corresponding norm $\|\cdot\|=\sqrt{\langle\cdot, \cdot\rangle}$ and $I$ denote the identity operator.

Let $T: H\to 2^H$ be a set-valued operator. $T^{-1}: H\to 2^H$ is inverse of $T$ which is defined by the relation $y\in Tx\Leftrightarrow x\in T^{-1}y$. The graph of $T$ is the set $\operatorname{gra}(T):=\{(x,y)\in H\times H:y\in Tx\}$. The resolvent of an operator $T$ of index $\gamma$ is defined by $J_{\gamma T}:= (I+\gamma T)^{-1}$, where $\gamma\in \mathbb{R}\setminus\{0\}$. The Yosida approximation of $T$ of index $\gamma$ is defined by $T_\gamma= (T^{-1}+ \gamma I)^{-1}$, $\gamma\in \mathbb{R}$. For $\gamma, \delta\in \mathbb{R}$, we have $(T_{\gamma})_\delta=T_{\gamma+\delta}$; in particular, $(T_{-\gamma})_\gamma= T$ (see \cite{attouch2018backward}).
\begin{definition}\cite{bauschke2011convex}
	A set-valued operator $T: H\to 2^H$ is said to be
	\begin{enumerate}
		\item [(i)] monotone if $$\left\langle x-y, u-v \right\rangle \geq 0~ \forall (x,u), (y,v)\in \operatorname{gra}(T);$$
		\item [(ii)] maximal monotone if there exist no monotone operator $S:H\rightarrow 2^H$ such that $\operatorname{gra}(S)$ properly contains $\operatorname{gra}(T)$;
		\item [(iii)] maximally $\rho$-cohypomonotone if $T_\rho= (T^{-1}+\rho I)^{-1}$ is maximally monotone, where $\rho\in \mathbb{R}$;
		\item [(iv)] uniformly monotone with modulus function $\phi: [0,\infty)\to [0,\infty)$, if $\phi$ is increasing, vanishes only at $0$, and $$ \langle x-y, u-v\rangle\ge \phi(\|x-y\|)~\forall(x,u), (y,v)\in \operatorname{gra}(T);$$
		\item[(v)] strongly monotone with constant $\rho\in [0,\infty)$ if $T-\rho I$ is monotone, i.e., $$ \langle x-y, u-v\rangle\ge \rho\|x-y\|^2~\forall(x,u), (y,v)\in\operatorname{gra} (T).$$
	\end{enumerate}
\end{definition}
\begin{remark}\normalfont
	On cohypomonotonicity, its useful and related notions see \cite{combettes2004proximal,rockafellar2009variational,iusem2003inexact}.
\end{remark}
\begin{example} \label{re}\normalfont
	(i)  Let $T:H\to 2^H$ be a maximally monotone operator. Then $T$ is $\rho$-cohypomonotone for all $\rho\ge 0$ (\cite{attouch2018backward}).\\
	
	(ii)
	Consider a bounded, linear and symmetric operator $N: H\to H$ whose spectrum $\sigma(N)$ has negative points. Define the set-valued operator $T:H\to 2^H$ by $T= N^{-1}$, which is not monotone. Then $T_\rho= (T+\rho I)^{-1}$ is maximal monotone operator for $\rho> -\min\sigma(N)$. Hence $T$ is maximally $\rho$-cohypomonotone (\cite{attouch2018backward}).
	
\end{example}
\begin{definition}\cite{bauschke2011convex}
	An operator $T: H\to H$ is said to be
	\begin{enumerate}
		\item [(i)]$\beta$-cocoercive for $\beta>0$ if $$\langle Tx- Ty, x-y\rangle\ge \beta \|Tx-Ty\|^2~ for~ every~ x,y\in H;$$
		\item [(ii)]non-expansive if $$\|Tx-Ty\|\le \|x-y\|~ for~ every~ x,y\in H;$$
		\item [(iii)]$\alpha$-averaged for $\alpha\in (0,1)$ if there exists a nonexpansive operator $R:H\to H$ such that $T=(1-\alpha)I+\alpha R$.
	\end{enumerate}
	\begin{remark}\label{r2.1}\normalfont
		If $T$ is a nonexpansive operator, then operator defined by $B=I-T$ is $\frac{1}{2}$-cocoercive.
	\end{remark}
	\begin{lemma}\label{l4}\cite{bauschke2011convex}
		Let $\beta>0$, $\gamma\in (0, 2\beta)$ and $T: H\to H$ be $\beta$-cocoercive. Then $I-\gamma T$ is $\frac{\gamma}{2\beta}$-averaged.
	\end{lemma}
	\begin{lemma}\label{l3}\cite{Ogura2002}
		Let $T_i: H\to H$ be $\alpha_i$-averaged operators for some $\alpha_i\in [0,1)$, where $i=1,2$. Then $\frac{\alpha_1+\alpha_2-2\alpha_1\alpha_2}{1-\alpha_1\alpha_2}\in [0,1)$ and $T_1T_2$ is $\frac{\alpha_1+\alpha_2-2\alpha_1\alpha_2}{1-\alpha_1\alpha_2}$-averaged.
	\end{lemma}
\end{definition}
\begin{lemma}\label{l2}\cite{attouch2018backward}
	Let $T: H\to 2^H$ be a set-valued operator, $\gamma\in \mathbb{R}$ and $\alpha>0$. Then $T$ is maximally $(\gamma-\alpha)$-cohypomonotone if and only if $T_\gamma$ is defined everywhere, single-valued and $\alpha$-cocoercive.
\end{lemma}


Let $M: H\to H$ be a bounded, linear, and self-adjoint operator. $M$ is said to be positive if $\langle Mx,x \rangle\ge 0$, for all $x\in H$, and strongly positive if $\exists ~m\in (0,\infty)$ such that $M-mI$ is positive. The square root and inverse of the strongly positive operator $M$ are denoted by $\sqrt{M}$ and $M^{-1}$.

The maps $(x,y)\mapsto\langle x,y \rangle_M:=\langle Mx,y \rangle$ and $x\mapsto \|x\|_M:= \sqrt{\langle Mx,x \rangle}$ define an inner product and a norm over $H$, respectively, where $M$ is a strongly positive operator on $H$. For all $x\in H$, $$m\|x\|^2\le \|x\|_M^2\le \|M\|\|x\|^2.$$ Thus, $\|\cdot\|$ and $\|\cdot\|_M$ are equivalent norms and hence induce the same topology over $H$.

For $\gamma\in \mathbb{R}\setminus\{0\}$ and strongly positive operator $M$, we denote $J_{\gamma T}^M:= (I+\gamma M^{-1}T)^{-1}$.
\begin{definition}\cite{bauschke2011convex}
	Let $f: H \to (-\infty, \infty]$ be a function.  
	\begin{itemize}
		\item [(i)] If $f$ is a proper function, then $f$ is said to be uniformly convex with modulus function $\phi:[0,\infty)\to [0,\infty)$ if  $$f(\nu x+ (1-\nu)y)+\nu(1-\nu)\phi(\|x-y\|)\le \nu f(x)+(1-\nu)f(y)$$~ for all $ \nu\in (0,1)$  and $x,y\in dom f$,  where function $\phi$ is increasing and vanishes at $0$.
		\item [(ii)] Let $\gamma>0$. The Moreau envelope of $f$ with parameter $\gamma$ is
		\begin{align*}
			f_\gamma(x)= \inf_{y\in H}\{f(y)+\frac{1}{2\gamma}\|x-y\|^2\}.
		\end{align*}
	\end{itemize}
\end{definition}
\begin{definition}\cite{attouch2011continuous,abbas2014newton}\label{def1}
	Let $b>0$, $x: [0,b]\to H$ be a function. Then $x$ is absolutely continuous if one of the following equivalent properties holds:
	\begin{itemize}
		\item [(i)] there exists an integrable function $y:[0,b]\to H$ such that
		\begin{align}
			x(t)=x(0)+\int_{0}^{t}y(s)ds~~~for ~all~ t\in [0,b];\nonumber
		\end{align}
		\item [(ii)] $x$ is continuous and its distribution derivative $\dot{x}$ is Lebesgue integrable on $[0,b]$;
		\item [(iii)] $\forall~ \epsilon >0$, $\exists~ \eta >0$ such that for any finite family of intervals $I_k=(a_k, b_k)$ we have the implication
		\begin{align}
			\left(I_k\cap I_j = \emptyset~ and~ \sum_{k}|b_k-a_k|< \eta\right) \Rightarrow \sum_{k} \|x(b_k)\|-x(a_k)\|< \epsilon.\nonumber
		\end{align}
	\end{itemize}
\end{definition}
\begin{remark}\label{Re1}\normalfont
	$(a)$ From Definition \ref{def1}, one asserts that an absolutely continuous function is differentiable almost everywhere, its derivative matches with its distributional derivative almost everywhere,  and one can achieve the function from its derivative $\dot{x}=y$ with the help of integration formula $(i)$.
	
	$(b)$ Given $b>0$, let $x:[0,b]\to H$ be an absolutely continuous function. Then one can show by using Definition $\ref{def1}(iii)$ that $z=B\circ x$ is absolutely continuous for $L$-Lipschitz continuous operator $B$. Also, $z$ is almost everywhere differentiable and $\|\dot{z}(\cdot)\|\le L\|\dot{x}(\cdot)\|$ holds almost everywhere.  
\end{remark}
\begin{definition}\label{d1}
	A map $u:[0,\infty)\to H$ is said to be strong solution of (\ref{meq}), if the following properties hold:
	\begin{itemize}
		\item [(i)] $u:[0,\infty)\to H$ is locally absolutely continuous;
		\item [(ii)] $\dot{u}(t)=\lambda (t)[(I-\gamma M^{-1}B)J_{\gamma A}^M u(t)-u(t)]$ for almost every $t\in [0,\infty)$;
		\item [(iii)] $u(0)= u_0.$
	\end{itemize}
\end{definition}
\begin{definition}
	A point $u^*$ is said to be an equilibrium point of dynamical system (\ref{meq}) if $u^*$ satisfies (\ref{12}), i.e., $0\in (A+B)u^*$.
\end{definition}
\begin{definition}\cite{nagurney2012projected}
	Let $u(t)$ be the solution of the dynamical system (\ref{meq}) and $u(0)=u_0$. Then an equilibrium point $u^*$ is said to be
	\begin{enumerate}
		\item [(i)]globally exponentially stable if there are constants $C_1>0$ and $C_2>0$ such that
		\begin{align*}
			\|u(t)-u^*\|\le C_1\|u_0-u^*\|\exp(-C_2t) ~\forall t>0;
		\end{align*}
		\item [(ii)]global monotone attractor if $\|u(t)-u^*\|$ is nonincreasing in $t$.
	\end{enumerate}
\end{definition}
\begin{lemma}\cite{abbas2014newton}\label{lm2.2}
	If $F:[0,\infty)\to [0,\infty)$ is locally absolutely continuous function, for $1\le p < \infty$, $1\le r\le \infty$, $F\in L^{p}\left([0,\infty)\right)$, $G:[0,\infty)\to \mathbb{R}$, $G\in L^{r}\left([0,\infty)\right)$ and for almost every $t\in [0,\infty)$ $$ \frac{d}{d(t)}F(t)\le G(t),$$
	then $\lim\limits_{t\to \infty}F(t)=0$.
\end{lemma}
\begin{lemma}\label{2l1}\cite{bot2016second}
	Let $C$ be a nonempty subset of $H$ and $u:[0,\infty)\to H$ ba a given map. Suppose that
	\begin{itemize}
		\item[(i)] $\lim\limits_{t\to\infty}\|u(t)-u^*\|$ exists, for every $u^*\in C$;
		\item[(ii)] every weak sequemtial cluster point of the map $u$ belongs to $C$.
	\end{itemize}
	Then there exists $u_\infty\in C$ such that $u(t)\rightharpoonup u_\infty$ as $t\to\infty$.
\end{lemma}
\begin{lemma}\cite{bauschke2011convex}\label{l2.4}
	Let $C\subseteq H$ be a nonempty closed convex set and $T: C\to H$ be a nonexpansive mapping. Let $\{u_n\}$ be a sequence in $C$ and $u\in H$. Suppose that $u_n\rightharpoonup u$ and that $u_n- Tu_n\to 0$. Then $u\in \operatorname{Fix (T)}$.
\end{lemma}

\begin{proposition}\label{l2.1}
	Let $T: H\to 2^H$ be a set-valued operator and $\gamma\in \mathbb{R}\setminus \{0\}$. Then we have the following:
	\begin{itemize}
		\item [(a)] $J_{\gamma T}^M= I-\gamma M^{-1} T_{\gamma}$.
		\item [(b)] $T_{\gamma}$ is defined everywhere and single-valued whenever $J_{\gamma T}^M$ is.
		\item [(c)] $J_{\gamma(T_{-\gamma})}^M= I-\gamma M^{-1} T$.
		\item [(d)] $T$ is defined everywhere and single-valued whenever $J_{\gamma (T_{-\gamma})}^M$ is.
	\end{itemize}
\end{proposition}
\begin{proof}
	(a) Since $\gamma \ne 0$, we obtain
	\begin{align}
		y\in J_{\gamma T}^Mx&\Leftrightarrow x\in y+\gamma M^{-1}Ty \nonumber\\[6pt]
		&\Leftrightarrow\frac{x-y}{\gamma}\in M^{-1}T(x-\gamma\frac{x-y}{\gamma})\nonumber\\[6pt]
		&\Leftrightarrow y\in x-\gamma M^{-1} T_\gamma x.\nonumber
	\end{align}
	(b) It follows from (a).\\
	(c) Replace $T$ by $T_{-\gamma}$ in (a), we get $$J_{\gamma (T_{-\gamma})}^M=I-\gamma M^{-1}(T_{-\gamma})_\gamma=I-\gamma M^{-1}T.$$
	(d) It follows from (c).
	
\end{proof}
\begin{proposition}\label{l1}
	Let $A: H\to 2^H$ be maximally $(\gamma-\alpha)$-cohypomonotone, $B: H\to H$ be $\beta$-cocoercive  and $M: H\to H$ be a strongly positive such that $\|M^{-1}\|\le \frac{1}{\kappa}\min\{\alpha, \beta\}$, where $\kappa>0$. Then we have the following:
	\begin{itemize}
		\item [(a)] $M^{-1}B$ is $\kappa$-cocoercive.
		\item [(b)] $M^{-1}A_{\gamma}$ is $\kappa$-cocoercive.
	\end{itemize}
\end{proposition}
\begin{proof} (a) Let $x, y\in H$.
	Since $B$ is $\beta$-cocercive, we have
	\begin{equation}
		\langle Bx-By, x-y\rangle \ge \beta\|Bx-By\|^2\\
		\Leftrightarrow \langle M^{-1}Bx-M^{-1}By, x-y\rangle_M\ge\beta\langle Bx-By, Bx-By\rangle.\nonumber
	\end{equation}
	Also, $\kappa\|M^{-1}Bx-M^{-1}By\|_M=\kappa\langle Bx-By, M^{-1}(Bx-By)\rangle$. Altogether, denoting $w:=(Bx-By)$, we obtain
	\begin{align}\label{1eq}
		\langle M^{-1}Bx-M^{-1}By, x-y\rangle_M-\kappa\|M^{-1}Bx-M^{-1}By\|_M \ge \beta\|w\|^2-\kappa\langle w, M^{-1}w\rangle.
	\end{align}
	By the Cauchy-Schwarz inequality, we have $$\kappa\langle w, M^{-1}w\rangle\le \kappa\|M^{-1}\|\|w\|^2.$$
	Hence, from (\ref{1eq}), we get
	\begin{align}
		\langle M^{-1}Bx-M^{-1}By, x-y\rangle_M\ge\kappa\|M^{-1}Bx-M^{-1}By\|_M,\nonumber
	\end{align}
	which implies that $M^{-1}B$ is $\kappa$-cocoercive.
	
	(b) From Lemma \ref{l2}, $A_\gamma$ is $\alpha$-cocoercive operator. So, in the similar manner, one can show $M^{-1}A_\gamma$ is $\kappa$-cocoercive.
\end{proof}
\begin{proposition}\label{lm1}
	Let $\gamma \neq 0.$ Let $A: H\to 2^H$ be a set-valued operator such that $J_{\gamma A}^M$ is single-valued and everywhere defined and $B: H\to H$ be an operator. Define $T_1:= (I+\gamma M^{-1}A)^{-1}(I-\gamma M ^{-1}B)$ and $T_2:= (I+\gamma M ^{-1}B_{-\gamma})^{-1}(I-\gamma M ^{-1}A_\gamma)$, where $M: H\to H$ is a strongly positive operator. Then the following statements hold:
	\begin{itemize}
		\item [(a)] $x\in \operatorname{Fix}(T_1)\Leftrightarrow x\in \operatorname{Zer} (A+B)\Leftrightarrow A_{\gamma}\circ(I-\gamma M^{-1}B)x+Bx=0$.
		\item [(b)] $y\in \operatorname{Fix}(T_2)\Leftrightarrow y\in \operatorname{Zer} (B_{-\gamma}+A_{\gamma})\Leftrightarrow B\circ J_{\gamma A}^M y+A_\gamma y=0$.
		\item [(c)] $I-\gamma M^{-1}B: \operatorname{Zer}(A+B)\to\operatorname{Zer}(B_{-\gamma}+A_{\gamma})$ is a bijection with inverse $J_{\gamma A}^M$.
	\end{itemize}
\end{proposition}
\begin{proof}
	(a) Suppose that $x\in \operatorname{Zer}(A+B)$. Then
	\begin{align}
		0\in (A+B)x
		&\Leftrightarrow 0\in \gamma M^{-1}Ax+\gamma M^{-1}Bx\nonumber\\[6pt]
		&\Leftrightarrow (I-\gamma M^{-1}B)x\in (I+\gamma M^{-1}A)x\nonumber\\[6pt]
		&\Leftrightarrow x= (I+\gamma M^{-1}A)^{-1}(I-\gamma M^{-1}B)x \nonumber\\[6pt]
		&\Leftrightarrow x\in \operatorname{Fix}(T_1).\nonumber
	\end{align}
	Also,
	\begin{align}
		x\in \operatorname{Fix}(T_1)&\Leftrightarrow x\in \operatorname{Fix}[(I-\gamma M^{-1}A_\gamma)(I-\gamma M^{-1}B)]\nonumber\\[6pt]
		&\Leftrightarrow x=x-\gamma M^{-1}Bx-\gamma M^{-1}A_{\gamma}\circ (I-\gamma M^{-1}B)x\nonumber\\[6pt]
		&\Leftrightarrow A_{\gamma}\circ(I-\gamma M^{-1}B)x+Bx=0.\nonumber
	\end{align}
	\item [(b)] In (a) apply $(B_{-\gamma}, A_\gamma)$ in place of $(A, B)$ and using Proposition \ref{l2.1}, we have the result.
	\item [(c)] Let $x\in \operatorname{Zer}(A+B)$. By (a), we have
	\begin{align}
		&A_\gamma \circ(I-\gamma M^{-1}B)x+B\circ J_{\gamma A}^M\circ(I-\gamma M^{-1}B)x=0\nonumber\\[6pt]
		&\Leftrightarrow (A_\gamma + B\circ J_{\gamma A}^M)\circ(I-\gamma M^{-1}B)x=0\nonumber\\[6pt]
		&\Leftrightarrow (I-\gamma M^{-1}B)x\in \operatorname{Zer}(B_{-\gamma}+A_{\gamma}).\nonumber
	\end{align}
	Also, for $y\in \operatorname{Zer}(B_{-\gamma}+A_{\gamma})$, from (b), we deduce
	\begin{align}
		B\circ J_{\gamma A}^M y+A_\gamma\circ(I-\gamma M^{-1}B)\circ J_{\gamma A}^My=0 &\Leftrightarrow (B+A_\gamma \circ(I-\gamma M^{-1}B))\circ J_{\gamma A}^M y=0\nonumber\\[6pt]
		&\Leftrightarrow J_{\gamma A}^M y \in \operatorname{Zer}(A+B).\nonumber
	\end{align}
	Finally, $J_{\gamma A}^M\circ(I-\gamma M^{-1}B)x=x$, and $(I-\gamma M^{-1}B)\circ J_{\gamma A}^M y=y$ for $x\in \operatorname{Zer}(A+B)$ and $y\in \operatorname{Zer}(B_{-\gamma}+A_{\gamma})$.
\end{proof}
\section{Convergence of trajectories generated by first order dynamical systems}\label{sc2}
\subsection{Operator Framework}
Consider the dynamical system
\begin{eqnarray}\label{2.10}{
		\left\{
		\begin{array}{lc@{}c@{}r}
			\dot{u}(t)=\lambda(t)[T(u(t))-u(t)]\\[6pt]
			u(0)= u_0,
		\end{array}\right.
}\end{eqnarray}
where $u_0\in H$, $T: H \to H$ is an $\alpha$-averaged operator and $\lambda:[0,\infty)\to [0,\infty)$ is a Lebesgue measurable function satisfying
\begin{align}\label{2.6}
	0<\underline{\lambda}\le \inf\limits_{t\ge 0}\lambda(t)\le \sup\limits_{t\ge 0}\lambda(t)\le \overline{\lambda},
\end{align}
where $\underline{\lambda}, \overline{\lambda}\in \mathbb{R}$.
First we establish the following result for the dynamical system (\ref{2.10}).
\begin{proposition}\label{prop2}
	Let $T:H\to H$ be an $\alpha$-averaged operator for $\alpha\in (0,1)$ with $\operatorname{Fix}(T)\neq \emptyset$. Let $u:[0,\infty)\to H$ be the unique strong global solution of the dynamical system (\ref{2.10}).
	Then we have the following:
	\begin{itemize}
		\item [(i)] The trajectory $u$ is bounded and $\dot{u}$, $(I-T)u\in L^2([0,\infty);H)$.
		\item [(ii)] $\lim\limits_{t\to \infty}\dot{u}(t)=\lim\limits_{t\to \infty}(I-T)(u(t))=0$.
		\item [(iii)] $u(t)\rightharpoonup \bar{u}\in \operatorname{Fix}(T)$ as $t\to \infty$.
	\end{itemize}
\end{proposition}
\begin{proof}
	(i) Let $u^*\in \operatorname{Fix}(T)$. Define $k(t):=\frac{1}{2}\|u(t)-u^*\|^2$, $t\in[0,\infty)$. Then $\dot{k}(t)=\langle u(t)-u^*, \dot{u}(t)\rangle$. From (\ref{2.10}), and the fact that $(I-T)u^*=0$, we have for every $t\in [0,\infty)$
	\begin{align}
		\dot{k}(t)+\lambda(t)\langle u(t)-u^*, (I-T)(u(t))- (I-T)u^*\rangle= 0. \label{2.5}
	\end{align}
	Since $T$ is $\alpha$-averaged operator, so there exists a nonexpansive operator $R:H\to H$ such that $T= (1-\alpha)I+\alpha R$ and $\operatorname{Fix}(T)= \operatorname{Fix}(R)$. From (\ref{2.5}) we have
	\begin{align}
		\dot{k}(t)+\alpha\lambda(t)\langle u(t)-u^*, (I-R)(u(t))- (I-R)u^*\rangle= 0.\label{2.5a}
	\end{align}
	Remark \ref{r2.1} shows that $I-R$ is $\frac{1}{2}$-cocoercive operator. From (\ref{2.5a}) we obtain
	\begin{align}
		\dot{k}(t)+\frac{\alpha\lambda(t)}{2}\|(I-R)(u(t))\|^2\le 0\nonumber,
	\end{align}
	which implies that
	\begin{align}
		\dot{k}(t)+\frac{\lambda(t)}{2\alpha}\|(I-T)(u(t))\|^2\le 0.\label{2.4a}
	\end{align}
	From (\ref{2.10}) and $(\ref{2.4a})$, we have
	\begin{align}
		\dot{k}(t)+\frac{1}{2\alpha\lambda(t)}\|\dot{u}(t)\|^2\le 0 ~for ~all~t\in [0, \infty).\nonumber
	\end{align}
	Using condition (\ref{2.6}), we get
	\begin{align}\label{2.7}
		\dot{k}(t)+\frac{1}{2\alpha\overline{\lambda}}\|\dot{u}(t)\|^2\le 0~ for~ every ~t\in [0,\infty).
	\end{align}
	From (\ref{2.7}), we get that the function $t\mapsto k(t)$ is monotonically decreasing. Also, the map $t\mapsto k(t)$ is locally absolutely continuous. Hence there exists $N_1\in \mathbb{R}$ such that
	$$k(t)\le N_1~for~all ~t\in [0,\infty).$$
	Thus, $k$ is bounded, and hence $u$ is bounded.\\
	To integrate the inequality (\ref{2.7}), we get that there is $N_2\in \mathbb{{R}}$ such that
	\begin{align}\label{2.8}
		k(t)+\frac{1}{2\alpha\overline{\lambda}}\int_{0}^{t} \|\dot{u}(t)\|^2\le N_2~~for~all~ t\in [0,\infty).
	\end{align}
	Since $k$ is bounded, so from (\ref{2.8}), we conclude that $\dot{u}(t)\in L^2([0,\infty);H)$. Hence, from $(\ref{2.10})$ and (\ref{2.6}), we get that $(I-T)u\in L^2([0,\infty);H)$.
	
	(ii) Using Remark \ref{Re1}(b) and cocoercivity of $I-R$, we have
	\begin{align}
		\frac{d}{dt}\left(\frac{1}{2}\|(I-T)(u(t))\|^2\right)&=\frac{d}{dt}\alpha^2\left(\frac{1}{2}\|(I-R)(u(t))\|^2\right)\nonumber\\&=\alpha^2\left\langle (I-R)(u(t)), \frac{d}{dt}((I-R)(u(t)))\right\rangle \nonumber\\
		&\le \frac{\alpha^2}{2}\|(I-R)(u(t))\|^2+2\alpha^2\|\dot{u}(t)\|^2\nonumber\\
		&= \frac{1}{2}\|(I-T)(u(t))\|^2+2\alpha^2\|\dot{u}(t)\|^2~~\forall t\in [0,\infty).\nonumber
	\end{align}
	By using Lemma \ref{lm2.2} and part (i) we obtain $\lim\limits_{t\to \infty}(I-T)(u(t))=0$, and from (\ref{2.6}) and (\ref{2.5}), we conclude that $\lim\limits_{t\to \infty} \dot{u}(t)=0.$
	
	(iii) We show that both the assumptions of Lemma \ref{2l1} are satisfied.\\
	As $k$ is bounded and $t\mapsto k(t)$ is monotonically decreasing, we observe that $\lim\limits_{t\to \infty}k(t)$ exists and belongs to the set of real numbers. So, $\lim\limits_{t\to \infty}\|u(t)-u^*\|$ exists.\\
	Let $\bar{u}$ be a weak sequential cluster point of $u(t)$, i.e, there exists a sequence $t_n\to \infty$ (as $n\to \infty$) such that $\{u(t_n)\}\rightharpoonup\bar{u}$. Using Lemma \ref{l2.4} and part (ii), we deduce that $\bar{u}\in \operatorname{Fix(R)}= \operatorname{Fix (T)}$ and the conclusion follows.
\end{proof}

In order to study the convergence behaviour of trajectories generated by dynamical systems (\ref{meq}) and (\ref{meq2}), we need the following assumptions:
\begin{itemize}
	\item [(A1)]The operator $A:H\to 2^H$ is maximally $(\gamma-\alpha)$-cohypomonotone.\label{a1}
	\item [(A2)]The operator $B:H\to H$ is $\beta$-cocoercive.\label{a2}
	\item [(A3)]$\operatorname{Zer}(A+B)\ne\emptyset$.
	\item [(A4)] The operator $M:H\to H$ is strongly positive.
\end{itemize}

Now we are ready to establish weak and strong convergence of trajectories generated by backward-forward first order dynamical system (\ref{meq}).
\begin{theorem}\label{th1}
	Let assumptions (A1), (A2), (A3) and (A4) hold and $\lambda:[0,\infty)\to [0,\infty)$ be a Lebesgue measurable function satisfying  condition (\ref{2.6}). Let $u:[0,\infty)\to H$ be the unique strong solution of (\ref{meq}) and $u_0\in H$. Let $\gamma\in (0,2\kappa)$, where $\kappa\in (0,\infty)$ such that
	\begin{align}\label{eq3.8}
		\kappa\le \|M\|\min\{\alpha,\beta\}.
	\end{align} Set $\delta:=\frac{2\kappa+\gamma}{2\gamma}$. Then the following statements hold:
	\begin{itemize}
		\item [(i)] The trajectory $u$ is bounded and $\dot{u}$, $(I-(I-\gamma M^{-1}B)\circ J_{\gamma A}^M )u\in L^2([0,\infty);H)$.
		\item[(ii)] $\lim\limits_{t\to\infty}\dot{u}(t)=\lim\limits_{t\to\infty}(I-(I-\gamma M^{-1}B)\circ J_{\gamma A}^M )(u(t))=0$.
		\item[(iii)] $u(t) \rightharpoonup u^*\in$ $\operatorname{Zer}(B_{-\gamma}+A_{\gamma})$ as $t\to \infty$.
		\item[(iv)] If $u^*\in \operatorname{Zer}(B_{-\gamma}+A_{\gamma})$, then $A_\gamma(u(\cdot))-A_{\gamma}u^*\in L^2([0,\infty);H)$ and $\lim\limits_{t\to \infty}A_{\gamma}(u(t))=A_{\gamma}u^*$.
		\item[(v)] $M^{-1}A_\gamma$ is constant on $\operatorname{Zer}(B_{-\gamma}+A_{\gamma})$.
		\item [(vi)] $v(t) \rightharpoonup v^*\in$ $\operatorname{Zer}(A+B)$ as $t\to \infty$, where $v(t)= J_{\gamma A}^M(u(t))$.
		\item [(vii)] $\lim\limits_{t\to\infty}B(v(t))= Bv^*= -A_{\gamma}u^*$.
		\item [(viii)] If $B_{-\gamma}$ or $A_{\gamma}$ is uniformly monotone, then $u(t)\to u^*\in \operatorname{Zer}(B_{-\gamma}+A_\gamma)$.
		\item [(ix)] If $B_{-\gamma}$ is $\rho$-strongly monotone for $\rho>0$ and choose $\eta>0$ fulfilling the condition:
		\begin{align}\label{9}
			\frac{1}{2\alpha}+\frac{\eta\|M\|^2}{2\gamma^2}\le \rho+\frac{\|M\|\overline{\lambda}}{\gamma}.
		\end{align}
		Let $u^*$ be an equilibrium point of dynamical system (\ref{meq}). Then we have the following:
		\begin{itemize}
			\item [(a)] If $\frac{1}{\eta\rho}<4$, then $u^*$ is globally exponentially stable.
			\item [(b)] If $\frac{1}{\eta\rho}=4$, then $u^*$ is global monotone attractor.
		\end{itemize}
	\end{itemize}
\end{theorem}
\begin{proof}
	(i)-(iii) We can write the dynamical system (\ref{meq}) in the form
	\begin{eqnarray}{
			\left\{
			\begin{array}{lc@{}c@{}r}
				\dot{u}(t)=\lambda (t)[T(u(t))-u(t)]\\[6pt]
				u(0)= u_0,
			\end{array}\right.
	}\end{eqnarray}
	where $T= (I- \gamma M^{-1}B)(I+\gamma M^{-1}A)^{-1}.$ From Proposition \ref{l2.1}, $T= (I-\gamma M^{-1}B)(I-\gamma M^{-1}A_\gamma)$. Since both $M^{-1}B$ and $M^{-1}A_\gamma$ are $\kappa$-cocoercive. From Lemma \ref{l4}, both $I- \gamma M^{-1}B$ and $I- \gamma M^{-1}A_\gamma$ are $\frac{\gamma}{2\kappa}$-averaged. Hence, by Lemma \ref{l3}, $T$ is $\frac{1}{\delta}$-averaged. Now, the statements (i)-(iii) follow from Propositions \ref{prop2} and \ref{lm1}, by observing that $\operatorname{Fix}(T)= \operatorname{Zer}(B_{-\gamma}+A_{\gamma})$.
	
	(iv) Let $u^*\in \operatorname{Zer}(B_{-\gamma}+A_{\gamma})$. From (\ref{meq}), we have
	\begin{align*}
		\frac{\dot{u}(t)}{\lambda(t)}+ u(t)= (I- \gamma M^{-1}B)(I-\gamma M^{-1}A_{\gamma})(u(t)),
	\end{align*}
	which implies that
	\begin{align*}
		-\frac{\dot{u}(t)}{\gamma\lambda(t)}- M^{-1}A_{\gamma}(u(t))= M^{-1}B(I-\gamma M^{-1}A_{\gamma})(u(t)).
	\end{align*}
	From Proposition \ref{lm1}(b), we have
	\begin{align*}
		-M^{-1} A_\gamma u^*= M^{-1}B(I-\gamma M^{-1}A_{\gamma})u^*.
	\end{align*}
	Since operators $M^{-1}B$ and $M^{-1}A_\gamma$ are $\kappa$-cocoercive, we deduce for every $t\in [0,\infty)$
	\begin{align}
		&\kappa\left\|-\frac{\dot{u}(t)}{\gamma\lambda(t)}- M^{-1}(A_{\gamma}(u(t))-A_{\gamma}u^*)\right\|^2\nonumber\\[6pt]
		&~~~\le\left\langle -\frac{\dot{u}(t)}{\gamma\lambda(t)}- M^{-1}(A_{\gamma}(u(t))-A_{\gamma}u^*), u(t)-u^*-\gamma(M^{-1}A_\gamma(u(t))-M^{-1}A_\gamma u^*)\right\rangle\nonumber\\[6pt]
		&~~~= \left\langle -\frac{\dot{u}(t)}{\gamma\lambda(t)}, u(t)-u^*-\gamma(M^{-1}A_\gamma(u(t))-M^{-1}A_\gamma u^*)\right\rangle \nonumber\\[6pt]
		&~~~-\langle M^{-1}A_\gamma(u(t))-M^{-1}A_\gamma u^*, u(t)-u^*\rangle+ \gamma\|M^{-1}A_\gamma(u(t))- M^{-1}A_\gamma u^*\|^2\nonumber\\[6pt]
		&~~~\le \left\langle -\frac{\dot{u}(t)}{\gamma\lambda(t)}, u(t)-u^*\right\rangle+ \gamma\left\langle \frac{\dot{u}(t)}{\gamma\lambda(t)}, M^{-1}A_\gamma(u(t))-M^{-1}A_\gamma u^*\right\rangle\nonumber\\[6pt]
		&~~~-\kappa\|M^{-1}A_\gamma(u(t))- M^{-1}A_\gamma u^*\|^2+\gamma\|M^{-1}A_\gamma(u(t))- M^{-1}A_\gamma u^*\|^2. \label{e3.7}
	\end{align}
	Also,
	\begin{align}
		\kappa\left\|-\frac{\dot{u}(t)}{\gamma\lambda(t)}- M^{-1}(A_{\gamma}(u(t))-A_{\gamma}u^*)\right\|^2=& \kappa\left\|\frac{\dot{u}(t)}{\gamma\lambda(t)}\right\|^2+\kappa\|M^{-1}A_\gamma(u(t))- M^{-1}A_\gamma u^*\|^2\nonumber\\[6pt]
		&+ 2\kappa\left\langle \frac{\dot{u}(t)}{\gamma\lambda(t)}, M^{-1}A_\gamma(u(t))- M^{-1}A_\gamma u^*\right\rangle.\label{e3.8}
	\end{align}
	From (\ref{e3.7}) and (\ref{e3.8}), we have
	\begin{align}
		&(2\kappa-\gamma)\|M^{-1}A_\gamma(u(t))- M^{-1}A_\gamma u^*\|^2\nonumber\\[6pt]
		&~~~ \le(\gamma-2\kappa)\left\langle \frac{\dot{u}(t)}{\gamma\lambda(t)},M^{-1}A_\gamma(u(t))- M^{-1}A_\gamma u^*\right\rangle\nonumber\\[6pt]
		&~~~-\left\langle \frac{\dot{u}(t)}{\gamma\lambda(t)}, u(t)-u^*\right\rangle- \frac{\kappa}{\gamma^2\lambda^2(t)}\|\dot{u}(t)\|^2\nonumber\\[6pt]
		&~~~\le \frac{(\gamma-2\kappa)}{2}\left\|\frac{\dot{u}(t)}{\gamma\lambda(t)}\right\|^2 +\frac{(\gamma-2\kappa)}{2}\|M^{-1}A_\gamma(u(t))- M^{-1}A_\gamma u^*\|^2\nonumber\\[6pt]
		&~~~-\left\langle \frac{\dot{u}(t)}{\gamma\lambda(t)}, u(t)-u^*\right\rangle\nonumber\\[6pt]
		& ~~~\le \frac{\gamma}{2}\left\|\frac{\dot{u}(t)}{\gamma\lambda(t)}\right\|^2+
		\frac{(\gamma-2\kappa)}{2}\|M^{-1}A_\gamma(u(t))- M^{-1}A_\gamma u^*\|^2-\frac{\kappa}{\gamma^2\lambda^2(t)}\|\dot{u}(t)\|^2\nonumber\\[6pt]&~~~-\left\langle \frac{\dot{u}(t)}{\gamma\lambda(t)}, u(t)-u^*\right\rangle\nonumber\\[6pt]
		& ~~~\le \frac{\gamma}{2}\left\|\frac{\dot{u}(t)}{\gamma\lambda(t)}\right\|^2+
		\frac{(\gamma-2\kappa)}{2}\|M^{-1}A_\gamma(u(t))- M^{-1}A_\gamma u^*\|^2-\left\langle \frac{\dot{u}(t)}{\gamma\lambda(t)}, u(t)-u^*\right\rangle.\nonumber
	\end{align}
	By using the function $k:[0,\infty)\to \mathbb{R}$, $k(t)=\frac{1}{2}\|u(t)-u^*\|^2$ and the fact that $\dot{k}(t)=\langle u(t)-u^*, \dot{u}(t)\rangle$, we obtain
	\begin{align*}
		\left((2\kappa-\gamma)-\frac{(2\kappa-\gamma)}{2}\right)\|M^{-1}A_\gamma(u(t))- M^{-1}A_\gamma u^*\|^2+\frac{1}{\gamma\lambda(t)}\dot{k}(t) \le&\frac{1}{2\gamma\lambda^2(t)}\left\| \dot{u}(t)\right\|^2.
	\end{align*}
	Taking into accounts the bounds of $\lambda$, we deduce for every $t\in [0, \infty)$
	\begin{align*}
		\left(\frac{2\kappa-\gamma}{2}\right)\|M^{-1}A_\gamma(u(t))- M^{-1}A_\gamma u^*\|^2+\frac{1}{\gamma\bar{\lambda}}\dot{k}(t)\le\frac{1}{2\gamma\overline{\lambda}^2}\left\| \dot{u}(t)\right\|^2.
	\end{align*}
	Integrating above equation from $0$ to $\tau$,  we get that for every $\tau\in [0, \infty)$
	\begin{align*}
		\left(\frac{2\kappa-\gamma}{2}\right)\int_{0}^{\tau}\|M^{-1}A_\gamma(u(t))- M^{-1}A_\gamma u^*\|^2 dt+\frac{1}{\gamma\bar{\lambda}}(k(\tau)-k(0))\nonumber\\\le \frac{1}{2\gamma\overline{\lambda}^2}\int_{0}^{\tau}\|\dot{u}(t)\|^2.
	\end{align*}
	Since $\dot{u}\in L^2([0,\infty); H)$, $\gamma\in(0,2\kappa)$, and $k(\tau)\ge 0$ for every $\tau\in [0, \infty)$, it follows that $M^{-1}A_\gamma(u(t))- M^{-1}A_\gamma u^*\in L^2([0,\infty); H)$ and hence $A_\gamma(u(t))- A_\gamma u^*\in L^2([0,\infty); H)$.\\
	From Remark \ref{Re1}, we get
	\begin{align}
		\frac{d}{dt} \left(\frac{1}{2}\|A_\gamma(u(t))-A_\gamma u^*\|\right)
		&= \left\langle A_\gamma(u(t))-A_\gamma u^*, \frac{d}{dt}\left(A_{\gamma}(u(t))\right) \right\rangle\nonumber\\[6pt]
		&\le \frac{1}{2}\|A_\gamma(u(t))-A_\gamma u^*\|^2+\frac{1}{2\alpha^2}\|\dot{u}(t)\|^2,\nonumber
	\end{align}
	and by Lemma \ref{lm2.2}, we have $\lim\limits_{t\to\infty}A_\gamma(u(t))=A_\gamma u^*$.
	
	(v) Let $u$ and $v$ are two zeros of $B_{-\gamma}+A_{\gamma}$, by Proposition \ref{lm1}, we obtain $-M^{-1}A_{\gamma}u= M^{-1}B(u-\gamma M^{-1}A_{\gamma}u)$ and $-M^{-1}A_{\gamma}v= M^{-1}B(v-\gamma M^{-1}A_{\gamma}v)$. Since $M^{-1}B$ is $\kappa$-cocoercive, we have
	\begin{align}
		&\kappa\|-M^{-1}A_\gamma v+M^{-1}A_\gamma u\|^2 \nonumber\\[6pt]
		&\le \langle -M^{-1}A_\gamma v+M^{-1}A_\gamma u, v-\gamma M^{-1}A_\gamma v- u+\gamma M^{-1}A_\gamma u\rangle\nonumber\\[6pt]
		&= \gamma\|-M^{-1}A_\gamma v+M^{-1}A_\gamma u\|^2-\langle -M^{-1}A_\gamma v+M^{-1}A_\gamma u, -v+u\rangle.\nonumber
	\end{align}
	By using $\kappa$-cocerciveness of $M^{-1}A_\gamma$, we obtain
	\begin{align*}
		\kappa\|-M^{-1}A_\gamma v+M^{-1}A_\gamma u\|^2 &\le -\kappa\|-M^{-1}A_\gamma v+M^{-1}A_\gamma u\|^2 \\
		&+\gamma\|-M^{-1}A_\gamma v+M^{-1}A_\gamma u\|^2.
	\end{align*}
	Since $\gamma <2\kappa$, so we get that $\|M^{-1}A_\gamma u- M^{-1}A_\gamma v\|^2=0$. Hence $M^{-1}A_\gamma$ is constant on $\operatorname{Zer}(B_{-\gamma}+A_{\gamma})$.
	
	(vi) From statements (iii), (iv) and Proposition \ref{l2.1}, we have $$v(t)= J_{\gamma A}^ M(u(t))= u(t)-\gamma M^{-1}A_{\gamma} (u(t))$$ converges weakly to $v^*= u^*-\gamma M^{-1} Bu^*$. From fixed point equality, we get $$u^*=(I-\gamma M^{-1}B)\circ J_{\gamma A}^M u^*= (I-\gamma M^{-1}B)v^*.$$ Enjoying the operator $J_{\gamma A}^ M$ with the equality $u^*= (I-\gamma M^{-1}B)v^*$ gives $v^*= J_{\gamma A}^ M\circ (I-\gamma M^{-1}B)v^*$, which declares that $v^*$ is fixed point of $J_{\gamma A}^ M\circ (I-\gamma M^{-1}B)$, hence a zero of $A+B$.
	
	(vii) Note,
	\begin{align*}
		M^{-1}B(v(t))&= \frac{1}{\gamma}\left(v(t)-J_{\gamma B_{-\gamma}}^M(v(t))\right)\\
		&= -M^{-1}A_{\gamma}u(t)+\frac{1}{\gamma}\left(v(t)-J_{\gamma B_{-\gamma}}^M(v(t))\right),
	\end{align*}
	which conclude that $B(v(t))\to -A_{\gamma}u^*$ as $t\to \infty$. Finally, from statement (vi) we deduce that $-A_{\gamma}u^*= Bv^*$.\newline
	(viii) Assume that $B_{-\gamma}$ is uniformly monotone with modulus function $\phi_{B_{-\gamma}}:[0,\infty)\to [0,\infty)$. Let $u^*\in \operatorname{Zer}(B_{-\gamma}+A_{\gamma})$ be an unique element.
	From Proposition \ref{l2.1}, we have for every $t\in [0,\infty)$
	\begin{align}\label{3.7}
		-M\frac{\dot{u}(t)}{\gamma\lambda(t)}- A_{\gamma}(u(t))\in B_{-\gamma}\left( \frac{\dot{u}(t)}{\lambda(t)}+u(t)\right).
	\end{align}
	Combining (\ref{3.7}) with $-A_{\gamma}u^*\in B_{-\gamma}u^*$, we have for every $t\in [0,\infty)$
	\begin{align*}
		&\phi_{B_{-\gamma}}\left(\left\|\frac{\dot{u}(t)}{\lambda(t)}+u(t)-u^*\right\|\right)\\[6pt]
		&\le \left\langle\frac{\dot{u}(t)}{\lambda(t)}+u(t)-u^*, A_\gamma u^*-M\frac{\dot{u}(t)}{\gamma\lambda(t)}-A_{\gamma}(u(t))\right\rangle\\[6pt]
		&=\left\langle\frac{\dot{u}(t)}{\lambda(t)},A_\gamma u^*-M\frac{\dot{u}(t)}{\gamma\lambda(t)} -A_{\gamma}(u(t))\right\rangle+\left\langle u(t)-u^*, M\frac{\dot{u}(t)}{\gamma\lambda(t)}\right\rangle\nonumber\\[6pt]
		&-\langle u(t)-u^*, A_{\gamma}(u(t))- A_{\gamma}u^*\rangle\nonumber,
	\end{align*}
	which combines with the monotonicity of $A_{\gamma}$ yields
	\begin{align}
		&\phi_{B_{-\gamma}}\left(\left\|\frac{\dot{u}(t)}{\lambda(t)}+u(t)-u^*\right\|\right)\nonumber\\[6pt]
		&\le\left\langle\frac{\dot{u}(t)}{\lambda(t)},A_\gamma u^* -A_{\gamma}(u(t))\right\rangle- \frac{1}{\lambda^2(t)\gamma}\|\dot{u}(t)\|_M^2+\left\langle u(t)-u^*, M\frac{\dot{u}(t)}{\gamma\lambda(t)}\right\rangle\nonumber\\[6pt]
		&\le\left\langle\frac{\dot{u}(t)}{\lambda(t)},A_\gamma u^* -A_{\gamma}(u(t))\right\rangle+\left\langle u(t)-u^*, M\frac{\dot{u}(t)}{\gamma\lambda(t)}\right\rangle.\label{8}
	\end{align}
	From parts (i)-(iv), (\ref{8}) and the fact that $\lambda$ is bounded by positive constants, we get
	\begin{align*}
		\lim\limits_{t\to \infty}\phi_{B_{-\gamma}}\left(\left\|\frac{\dot{u}(t)}{\lambda(t)}+u(t)-u^*\right\|\right)=0.
	\end{align*}
	Since the function $\phi_{B_{-\gamma}}$ is increasing vanishes to 0, so we have $$\left(\frac{\dot{u}(t)}{\lambda(t)}+u(t)-u^*\right)\to 0 ~\text{as}~ t\to \infty.$$ Using statement (ii) and the boundedness of $\lambda$ we get that $u(t)$ converges strongly to $u^*$ as $t\to\infty$.
	
	(ix) Suppose that $B_{-\gamma}$ is $\rho$-strong monotonicity. Combining (\ref{3.7}) with $-A_\gamma u^*\in B_{-\gamma}u^*$, we have
	\begin{align}
		&\rho\left\|\frac{\dot{u}(t)}{\lambda(t)}+u(t)-u^*\right\|^2\nonumber\\
		&\le \left\langle \frac{\dot{u}(t)}{\lambda(t)}+u(t)-u^*,A_{\gamma}u^* -\frac{M}{\gamma \lambda(t)}\dot{u}(t)-A_{\gamma}(u(t))\right\rangle \nonumber\\[6pt]
		& = \left\langle \frac{\dot{u}(t)}{\lambda(t)}, A_{\gamma}u^* -\frac{M}{\gamma \lambda(t)}\dot{u}(t)-A_{\gamma}(u(t))\right\rangle+ \left\langle u(t)-u^*, A_{\gamma}u^* -\frac{M}{\gamma \lambda(t)}\dot{u}(t)-A_{\gamma}(u(t))\right\rangle\nonumber\\[6pt]
		&\le \frac{1}{2\alpha\lambda^2(t)}\|\dot{u}(t)\|^2+\frac{\alpha}{2}\|A_\gamma u^*- A_{\gamma}(u(t))\|^2-\frac{1}{\gamma\lambda(t)}\|\dot{u}(t)\|_M^2+ \left\langle  u(t)-u^*, \frac{-M}{\gamma \lambda(t)}\dot{u}(t) \right\rangle \nonumber\\[6pt]
		&-\langle u(t)-u^*, A_{\gamma}(u(t))- A_{\gamma}u^*\rangle. \nonumber
	\end{align}
	Using $\alpha$-cocoercivity of $A_\gamma$, we have
	\begin{align}
		&\rho\left\|\frac{\dot{u}(t)}{\lambda(t)}+u(t)-u^*\right\|^2\nonumber\\
		&\le \frac{1}{2\alpha\lambda^2(t)}\|\dot{u}(t)\|^2-\frac{1}{\gamma\lambda(t)}\|\dot{u}(t)\|_M^2+ \left\langle  u(t)-u^*, \frac{-M}{\gamma \lambda(t)}\dot{u}(t) \right\rangle \nonumber\\[6pt]
		&\le \frac{1}{2\alpha\lambda^2(t)}\|\dot{u}(t)\|^2-\|M\| \frac{1}{\gamma\lambda(t)}\|\dot{u}(t)\|^2+\frac{1}{2\eta}\|u(t)-u^*\|^2+\frac{\eta\|M\|^2}{2\gamma^2\lambda^2(t)}\|\dot{u}(t)\|^2\nonumber\\[6pt]
		&= \left( \frac{1}{2\alpha\lambda^2(t)}- \frac{\|M\|}{\gamma\lambda(t)}+\frac{\eta\|M\|^2}{2\gamma^2\lambda^2(t)} \right)\|\dot{u}(t)\|^2 +\frac{1}{2\eta}k(t).\label{3.6}
	\end{align}
	Also,
	\begin{align}\label{3.8}
		\rho\left\|\frac{\dot{u}(t)}{\lambda(t)}+u(t)-u^*\right\|^2&= \frac{\rho}{\lambda^2(t)}\|\dot{u}(t)\|^2+\frac{2\rho}{\lambda(t)}\dot{k}(t)+2\rho k(t).
	\end{align}
	From (\ref{3.6}) and (\ref{3.8}), we have for $t\in [0,\infty)$
	\begin{align}
		\frac{2\rho}{\lambda(t)}\dot{k}(t)+\left(2\rho-\frac{1}{2\eta}\right)k(t)+\left( \frac{\rho}{\lambda^2(t)}-\frac{1}{2\alpha\lambda^2(t)}+ \frac{\|M\|}{\gamma\lambda(t)}-\frac{\eta\|M\|^2}{2\gamma^2\lambda^2(t)} \right)\|\dot{u}(t)\|^2\le 0.\nonumber
	\end{align}
	So, from (\ref{9}) and the fact that $\lambda(t)$ is bounded, we have for every $t\in[0, \infty)$
	\begin{align}
		\frac{2\rho}{\underline{\lambda}}\dot{k}(t)+\left(2\rho-\frac{1}{2\eta}\right)k(t)\le 0,\nonumber
	\end{align}
	which implies that
	\begin{align}
		\dot{k}(t)+\frac{2\rho-\frac{1}{2\eta}}{\frac{2\rho}{\underline{\lambda}}}k(t)\le 0.\nonumber
	\end{align}
	Now we have two cases:
	\begin{enumerate}
		\item [(a)] If $\frac{1}{\eta\rho}< 4$, then we have
		\begin{align}\label{3.9}
			\dot{k}(t)+Ck(t)\le 0,
		\end{align}
		for every $t\in[0,\infty)$, where $C:=\underline{\lambda}\left( 1-\frac{1}{4\eta\rho}\right)>0$.
		Now, by multiplying $\exp(Ct)$ with (\ref{3.9}) and integrating between $0$ and $\tau$, where $\tau\ge 0$, we have the desired result.
		\item [(b)] If $\frac{1}{\eta\rho}=4$, then $\dot{k}(t)\le 0$, so $u^*$ is global monotone attractor.
	\end{enumerate}
\end{proof}
We now study the convergence of trajectories generated by forward-backward first order dynamical system (\ref{meq2}). Before analyzing the results based on the dynamical system (\ref{meq2}), we need the following proposition.
\begin{proposition}\label{pro1}
	Let $\gamma\neq 0$ and $\lambda:[0,\infty)\to[0,\infty) $ be a function. Let $A:H\to 2^H$ be a set-valued operator such that $J_{\gamma A}^M$ is everywhere defined and single-valued and $B: H\to H$. Then we have the following:
	\begin{itemize}
		\item [(i)] For each $u_0\in H$, dynamical system (\ref{meq}) applied to $(A,B)$ is uniquely defines a function $u(t)$. For $x_0\in H$, dynamical system (\ref{meq2}) applied to $(B_{-\gamma}, A_{\gamma})$ is uniquely defines a function $x(t)$. Moreover, if $u_0= x_0$, then $u(t)= x(t)$ for $t\in [0, \infty)$.
		\item [(ii)] For each $x_0\in H$, dynamical system (\ref{meq2}) applied to $(A,B)$ is uniquely defines a function $x(t)$. For $u_0\in H$, dynamical system (\ref{meq}) applied to $(B_{-\gamma}, A_{\gamma})$ is uniquely defines a function $u(t)$. Moreover, if $u_0= x_0$, then $u(t)= x(t)$ for $t\in [0, \infty)$.
	\end{itemize}
\end{proposition}
\begin{proof}
	(i) The uniqueness of the dynamical system (\ref{meq2}) follows from the hypothesis (as $A_\gamma$ and $J_{\gamma B_{-\gamma}}^M$ are single-valued and everywhere defined). If we consider $T= B_{-\gamma}$ in Proposition \ref{l2.1}(ii) and $T= A_{\gamma}$ in Proposition \ref{l2.1}(a), then the system (\ref{meq}) is uniquely defined and satisfies the same relation as $(\ref{meq2})$ does.
	
	(ii) It follows from (i), by using $(A, B)$ in place of $(B_{-\gamma}, A_{\gamma})$.
\end{proof}

\begin{theorem}\label{th3}
	Let assumptions (A1), (A2), (A3) and (A4) hold and $\lambda:[0,\infty)\to [0,\infty)$ be a Lebesgue measurable function satisfying condition (\ref{2.6}). Let $x:[0,\infty)\to H$ be the unique strong solution of (\ref{meq2}) and $x_0\in H$. Let $\gamma\in (0,2\kappa)$, where $\kappa\in (0,\infty)$ satisfying (\ref{eq3.8}). Set $\delta:=\frac{2\kappa+\gamma}{2\gamma}$. Then the following statements hold:
	\begin{itemize}
		\item [(i)] The trajectory $x$ is bounded and $\dot{x}$, $(I- J_{\gamma A}^M\circ(I-\gamma M^{-1}B) )x\in L^2([0,\infty);H)$.
		\item[(ii)] $\lim\limits_{t\to\infty}\dot{x}(t)=\lim\limits_{t\to\infty}(I- J_{\gamma A}^M\circ(I-\gamma M^{-1}B))(x(t))=0$.
		\item[(iii)] $x(t) \rightharpoonup x^*\in$ $\operatorname{Zer}(A+B)$ as $t\to \infty$.
		\item[(iv)] If $x^*\in \operatorname{Zer}(A+B)$, then $B(x(\cdot))-Bx^*\in L^2([0,\infty);H)$, $\lim\limits_{t\to \infty}B(x(t))=Bx^*$.
		\item [(v)] $y(t) \rightharpoonup y^*\in$ $\operatorname{Zer}(A+B)$ as $t\to \infty$, where $y(t)= (I-\gamma M^{-1}B)(x(t))$.
		\item [(vi)] $\lim\limits_{t\to\infty}A_{\gamma}(y(t))= A_{\gamma}y^*= -Bx^*$.
		\item [(vii)] If $A$ or $B$ is uniformly monotone, then $x(t)$ converges strongly to a unique point in $\operatorname{Zer}(A+B)$ as $t\to\infty$.
		\item [(viii)] If $A$ is $\rho$-strongly monotone for some $\rho>0$ and $\eta>0$ such that:
		\begin{align*}
			\frac{1}{2\alpha}+\frac{\eta\|M\|^2}{2\gamma^2}\le \rho+\frac{\|M\|\overline{\lambda}}{\gamma}.
		\end{align*}
		Let $x^*$ be an equilibrium point of dynamical system (\ref{meq2}). Then we have the following:
		\begin{itemize}
			\item [(a)] If $\frac{1}{\eta\rho}< 4$, then $x^*$ is globally exponentially stable.
			\item [(b)] If $\frac{1}{\eta\rho}= 4$, then $x^*$ is global monotone attractor.
		\end{itemize}
	\end{itemize}
\end{theorem}
\begin{proof}
	By Proposition \ref{pro1}, dynamical system (\ref{meq2}) applied to $(B_{-\gamma}, A_{\gamma})$ is system (\ref{meq}) applied to $(A, B)$. Since the pair $(A, B)$ satisfies the same assumption as $(B_{-\gamma}, A_{\gamma})$, so by using $(A, B)$ instead of $(B_{-\gamma}, A_{\gamma})$ in Theorem \ref{th1}, we have the desired result.
\end{proof}
\subsection{Function framework}
In this section, to study the optimization problem (\ref{eq1}) we consider the following first order dynamical system
\begin{eqnarray}\label{1}{
		\left\{
		\begin{array}{lc@{}c@{}r}
			\dot{u}(t)= \lambda(t)[(I-\gamma M^{-1}\nabla g) \operatorname{prox}_{\gamma f}^M ~u(t)-u(t)]\\[6pt]
			u(0)=u_0,
		\end{array}\right.
}\end{eqnarray}
where $u_0\in H$ and $\lambda:[0, \infty)\to [0, \infty)$ is Lebesgue measurable function.

To study the convergence behavior of trajectories generated by the dynamical system (\ref{1}), we need the following assumptions:
\begin{itemize}
	
	\item [(B1)] $f:H\to \mathbb{R}\cup\{\infty\}$ is proper, convex, and lower-semicontinuous function.
	\item [(B2)] $g:H\to \mathbb{R}$ is differentiable and its gradient $\nabla g$ is $\beta$-cocoercive function.
	\item [(B3)]$\operatorname{Argmin}(f+g)\ne\emptyset.$
\end{itemize}

\begin{remark}\label{re1}
	\normalfont{
		\item [(1)] Considering $\alpha=\gamma$, $A=\partial f$, $B=\nabla g$, assumptions (A1) and (A2) hold.
		Since $ A$ is maximal monotone, the operator $A_\gamma$ is defined everywhere, single-valued and $\gamma$-cocoercive \cite{bauschke2011convex}. Moreover, by the Moreau-Rockafellar theorem \cite{peypouquet2015convex}, one has $\partial(f+g)=\partial f+\nabla g$, so $\operatorname{Zer}(A+B)=\operatorname{Zer}(\partial f+ \nabla g)= \operatorname{Zer}(\partial(f+g))= \text{Argmin}(f+g)\ne \emptyset$.
		\item [(2)] If Moreau envelope of $f$ is denoted by $f_\gamma$, then $\nabla f_\gamma= (I-\operatorname{prox}_{\gamma f})/\gamma$ \cite{bauschke2011convex}. So, $\nabla f_\gamma= A_\gamma$.
		\item [(3)] The proximity operator of a proper, lower semicontinuous and convex function $f$ relative to the metric induced by strongly positive operator $M$ is (see \cite{combettes2014variable})
		\begin{align}
			\operatorname{prox}_{\gamma f}^M(u) = \underset{v\in H}{\operatorname{Argmin}}\left\lbrace f(v)+\frac{1}{2\gamma}\|v-u\|^2_M\right\rbrace. \nonumber
		\end{align}
		Note that $J_{\gamma \partial f}^M=\operatorname{prox}_{\gamma f}^M= (I+\gamma M^{-1}\partial f)^{-1}$.}
\end{remark}
\begin{theorem}\label{th2}
	Let assumptions (B1), (B2), (B3) and (A4) hold, $\lambda:[0,\infty) \to [0,\infty)$ be a Lebesgue measurable  function satisfying condition (\ref{2.6}), $u_0\in H$ and $u:[0,\infty)\to H$ be the unique strong global solution of dynamical system (\ref{1}). Let $\gamma\in (0,2\kappa)$, where $0<\kappa\le \|M\|\beta$ and $\|M\|\ge \frac{1}{2}$. Set $\delta:=\frac{2\kappa+\gamma}{2\gamma}$. Then the following statements hold:
	\begin{itemize}
		\item [(i)] The trajectory $u$ is bounded and $\dot{u}$, $(I-(I-\gamma M^{-1}\nabla g) \operatorname{prox}_{\gamma f}^M)u\in L^2([0,\infty);H)$.
		\item[(ii)] $\lim\limits_{t\to\infty}\dot{u}(t)=\lim\limits_{t\to\infty}(I-(I-\gamma M^{-1}\nabla g) \operatorname{prox}_{\gamma f}^M)(u(t))=0$.
		\item[(iii)] $u(t)\rightharpoonup u^*$.
		\item[(iv)] $\lim\limits_{t\to \infty}\nabla f_{\gamma}(u(t))=\nabla f_{\gamma}(u^*)$.
		\item[(v)] $v(t)\rightharpoonup v^*\in \operatorname{Argmin}(f+g)$, where $v(t)= \operatorname{prox}_{\gamma f}^M (u(t))$.
		\item[(vi)] $\lim\limits_{t\to \infty}\nabla g(v(t))= \nabla g(v^*)= -\nabla f_\gamma(u^*)$.
		\item [(vii)] $(f+g)(v(t))\to (f+g)(v^*)= \inf(f+g)$, $f(v(t))\to f(v^*)$ and $g(v(t))\to g(v^*)$.
	\end{itemize}
\end{theorem}
\begin{proof}
	The statements (i) through (vi) follow from Theorem \ref{th1} and Remark \ref{re1}, by taking $A:=\partial f$ and $B:= \nabla g$.
	
	(vii) First, we show that $(f+g)(v(t))\to (f+g)(v^*)$. Since $(f+g)$ is lower semicontinuity and $v(t)\rightharpoonup v^*$, we have $(f+g)(v^*)\le \liminf(f+g)(v(t))$. Suppose that $w(t)=v(t)-\gamma M^{-1}\nabla g(v(t))$ and since $v(t)=\operatorname{prox}_{\gamma f}^M(u(t))$, we have
	\begin{align}\label{3.20}
		u(t)\in v(t)+\gamma M^{-1}\partial f(v(t))&= w(t)+\gamma M^{-1}\nabla g(v(t))+\gamma M^{-1}\partial f(v(t))\\[6pt]
		&= w(t)+\gamma M^{-1}\partial (f+g)(v(t)).\nonumber
	\end{align}
	So, from the definition of subgradient of $f+g$ at $v(t)$
	\begin{align*}
		(f+g)(v^*)\ge (f+g)(v(t))+\langle v^*-v(t), \frac{M}{\gamma}(u(t)-w(t))\rangle.  
	\end{align*}
	In view of (ii), $(f+g)(v^*)\ge \limsup (f+g)(v(t))$. So, we have $(f+g)(v(t))\to (f+g)(v^*)$.\\
	Secondly, we show that $f(v(t))\to f(v^*)$ as $t\to \infty$. From statement (v) and the fact that $f$ is lower semicontinuous, we obtain $f(v^*)\le \liminf f(v(t))$. Also, from (\ref{3.20}) and by the definition of subgradient of $f$ at point $v(t)$, we have
	\begin{align}
		f(v^*) &\ge f(v(t))+\langle v^*-v(t), \frac{M}{\gamma}(u(t)-w(t)-\nabla g(v(t)))\rangle \nonumber\\
		&\ge f(v(t))+\langle v^*-v(t), \frac{M}{\gamma}(u(t)-w(t)-(\nabla g(v(t))-\nabla g(v^*)))\rangle \nonumber\\
		&+\langle v^*-v(t), -\nabla g(v^*)\rangle\nonumber.
	\end{align}
	From statements (ii), (v) and (vi), we have $f(v^*)\ge \limsup f(v(t))$, which shows our desired result.
\end{proof}
Suppose $\gamma \in \mathbb{R}$. Define the function $l: H\to \mathbb{R}\cup\{\infty\}$, by $l_{\gamma}^*= \left(l^*+\frac{\gamma}{2}\|\cdot\|^2\right)$, where $l_\gamma^*$ is the Frenchel conjugate of a function $l_\gamma$ \cite{bauschke2011convex}. So, $l_\gamma= \left(l^*+\frac{\gamma}{2}\|\cdot\|^2\right)^*$. Since, for $\gamma >0$ and convex $l$, the function $\left(l^*+\frac{\gamma}{2}\|\cdot\|^2\right)^*$ is the Moreau envelope of $l$, so the notation of Frenchel conjugate is compatible with the notation for the Moreau envelope.
\begin{lemma}\label{l3.1}\cite{attouch2018backward}
	Let $\gamma\in (0, \beta]$ and $g:H\to \mathbb{R}$ be convex and differentiable. Let $\nabla g$ be $\beta$-cocoercive. Then the following statements hold:
	\begin{itemize}
		\item [(i)] For $\gamma\ge -\beta$, $\mu\in \mathbb{R}$, we have $(g_\gamma)_\mu=g_{\gamma+\mu}$. In particular, $(g_{-\gamma})_\gamma=g$.
		\item [(ii)] $g_{-\gamma}$ is convex, lower semicontinious, proper and $\operatorname{prox}_{\gamma g_{-\gamma}}^M= I-\gamma M^{-1}\nabla g$.
		\item [(iii)] For all $u\in H$, $g_{-\gamma}(u)=\sup\limits_{\eta\in H}\left\lbrace g(\eta)-\frac{1}{2\gamma}\|u-\eta\|^2\right\rbrace$.
	\end{itemize}
\end{lemma}
\begin{lemma}\label{l3.2}
	Let $\gamma\in (0, \beta]$, and assumptions (B1), (B2), (B3) and (A4) hold. Then
	\begin{enumerate}
		\item [(i)] $I-M^{-1}\gamma \nabla g: \operatorname{Argmin}(f+g)\to \operatorname{Argmin}(f_{\gamma}+g_{-\gamma})$ is a bijection with inverse $\operatorname{prox}_{\gamma f}^M$.
		\item [(ii)] $\inf(f+g)= \inf(g_{-\gamma}+f_{\gamma})$.
	\end{enumerate}
\end{lemma}
\begin{proof}
	Proof follows from [\cite{attouch2018backward}, Proposition 4.6], Remark \ref{re1}, Lemma \ref{l3.1} and Lemma \ref{lm1}.
\end{proof}
In Theorem \ref{th2}, we have discussed the asymptotic convergence of the trajectories of (\ref{1}) under the condition $\gamma\in (0,2\kappa);~ \kappa>0$. In convex optimization, the interesting observation is to think  about the choice of step sizes. Now, we choose $\gamma\in(0,\beta]$ and discuss the dynamical system (\ref{1}) of the optimization problem.
\begin{theorem}
	Let assumptions (B1), (B2), (B3) and (A4) hold. Let $\gamma\in (0,\beta]$, $\lambda:[0,\infty)\to [0,\infty)$ be a Lebesgue measurable function satisfying condition (\ref{2.6}), $u_0\in H$ and $u:[0,\infty)\to H$ be the unique strong global solution of dynamical system (\ref{1}). Then the statements of Theorem \ref{th2} and the following statements are true:
	\begin{itemize}
		\item[(i)] $u(t)$ converges weakly to a minimizer of $g_{-\gamma}+f_\gamma$ as $t\to \infty$.
		\item[(ii)] If $u^*$ is a minimizer of $g_{-\gamma}+f_\gamma$, then $\nabla f_\gamma(u(\cdot))-\nabla f_{\gamma}(u^*)\in L^2([0,\infty);H)$, $\lim\limits_{t\to \infty}\nabla f_{\gamma}(u(t))=\nabla f_{\gamma}(u^*)$ and $\nabla f_\gamma(u(t))$ is constant on $g_{-\gamma}+f_{\gamma}$.
		\item[(iii)] If $\partial g_{-\gamma}$ or $\nabla f_{\gamma}$ is uniformly convex, then $u(t)$ converges strongly to a minimizer of $g_{-\gamma}+f_{\gamma}$ as $t\to\infty$.
		\item[(iv)]  If $\partial g_{-\gamma}$ is $\rho$-strongly convex for $\rho>0$, and choose $\eta>0$ fulfilling the condition:
		\begin{align*}
			\frac{1}{2\alpha}+\frac{\eta\|M\|^2}{2\gamma^2}\le \rho+\frac{\|M\|\overline{\lambda}}{\gamma}.
		\end{align*}
		Let $u^*$ be an equilibrium point of dynamical system (\ref{1}). Then we have the following:
		\begin{itemize}
			\item [(a)] If $\frac{1}{\eta\rho}<4$, then $u^*$ is globally exponentially stable.
			\item [(b)] If $\frac{1}{\eta\rho}=4$, then $u^*$ is global monotone attractor.
		\end{itemize}
	\end{itemize}
\end{theorem}
\begin{proof}
	It follows by applying Lemma \ref{l3.1} and Lemma \ref{l3.2} to Theorem \ref{th1}.
\end{proof}
Now, we discuss the convergence of the trajectories of (\ref{1}) without the restriction on the choice of the step size $(\gamma\in (0,2\kappa);~\kappa>0)$.
\begin{theorem}\label{Thm4}
	Let assumptions (B1), (B2), (B3) and (A4) hold. Let $\lambda:[0,\infty)\to [0,\infty)$ be a Lebesgue measurable function satisfying condition (\ref{2.6}), $u_0\in H$ and $u:[0,\infty)\to H$ be the unique strong global solution of (\ref{1}). Then we have the following:
	\begin{itemize}
		\item [(i)] The trajectory $u$ is bounded and $\dot{u}$, $(I-(I-\gamma M^{-1}\nabla g) \operatorname{prox}_{\gamma f}^M)u\in L^2([0,\infty);H)$.
		\item[(ii)] $\lim\limits_{t\to\infty}\dot{u}(t)=\lim\limits_{t\to\infty}(I-(I-\gamma M^{-1}\nabla g) \operatorname{prox}_{\gamma f}^M)(u(t))=0$.
		\item[(iii)] $u(t)\rightharpoonup u^*\in\operatorname{Zer}((\nabla g)_{-\gamma}+\nabla f_\gamma)$ as $t\to \infty$.
		\item[(iv)] If $u^*\in \operatorname{Zer}((\nabla g)_{-\gamma}+\nabla f_\gamma)$, then $\nabla f_\gamma(u(\cdot))-\nabla f_{\gamma}(u^*)\in L^2([0,\infty);H)$.
		\item[(v)] If $(\nabla g)_{-\gamma}$ or $\nabla f_{\gamma}$ is uniformly convex, then $u(t)\to u^*\in \operatorname{Zer}((\nabla g)_{-\gamma}+\nabla f_\gamma)$ as $t\to\infty$.
	\end{itemize}
\end{theorem}
\begin{proof}
	Let $u^* \in \operatorname{Zer}((\nabla g)_{-\gamma}+\nabla f_\gamma)$. From the definition of proximal operator and Proposition \ref{l2.1}, we have for every $t\in [0,\infty)$
	\begin{align}
		-M\frac{\dot{u}(t)}{\gamma\lambda(t)}- \nabla f_\gamma= (\nabla g)_{-\gamma}\left( \frac{\dot{u}(t)}{\lambda(t)} + u(t)\right). \label{1.23}
	\end{align}Combining (\ref{1.23}) with $-\nabla f_\gamma (u^*)\in (\nabla g)_{-\gamma}(u^*)$ and using the maximal monotonicity of $(\nabla g)_{-\gamma}$, we have for every $t\in [0,\infty)$
	\begin{align*}
		\left\langle \frac{\dot{u}(t)}{\lambda(t)}+u(t)-u^*, \nabla f_{\gamma}(u^*)- M\frac{\dot{u}(t)}{\gamma\lambda(t)}- \nabla f_{\gamma}(u(t))\right\rangle\ge 0.
	\end{align*}
	Since $\nabla f_{\gamma}$ is $\alpha$-cocoercive, so for every $t\in [0,\infty)$
	\begin{align*}
		\alpha\|\nabla f_\gamma (u(t))- \nabla f_\gamma (u^*)\|&\le \langle u(t)-u^*, \nabla f_{\gamma}(u(t))- \nabla f_{\gamma}(u^*)\rangle\\[6pt]
		&\le \left\langle  \frac{\dot{u}(t)}{\lambda(t)}, \nabla f_{\gamma}- M\frac{\dot{u}(t)}{\gamma\lambda(t)}-\nabla f_{\gamma}(u(t))\right \rangle\\[6pt]
		&+ \left\langle  u(t)-u^*,-M\frac{\dot{u}(t)}{\gamma\lambda(t)} \right\rangle\\[6pt]
		&\le \frac{1}{\lambda(t)}\langle \dot{u}(t), \nabla f_\gamma (u^*)-\nabla f_\gamma(u(t))\rangle- \frac{1}{\gamma\lambda^2(t)}\|\dot{u}(t)\|_M^2\\[6pt]
		&+\frac{1}{\gamma\lambda(t)}\langle u(t)-u^*, -M\dot{u}(t)\rangle.
	\end{align*}
	Define the map $q:[0,\infty)\to \mathbb{R}$, $q(t)=f_\gamma(u(t))-f_\gamma(u^*)-\langle \nabla f_\gamma(u^*),u(t)-u^*\rangle$.\\
	Since $\nabla f_\gamma$ is convex function, so we have $$q(t)\ge 0 ~\forall t\ge 0.$$
	Also, for any $t\in [0,\infty)$
	$$ \dot{q}(t)=\langle \dot{u}(t), \nabla f_\gamma(u(t))-\nabla f_\gamma(u^*)\rangle.$$
	Define the function $h: [0,\infty)\to \mathbb{R}, h(t)= \frac{1}{2}\|u(t)-u^*\|^2_M$ and using the fact that $\dot{h}(t)=\langle u(t)-u^*, M\dot{u}(t)\rangle $, we obtain
	\begin{align}\label{1.1}
		\alpha\lambda(t)\|\nabla f_\gamma(u(t))-\nabla f_\gamma(u^*)\|\le -\frac{d}{dt}(q(t))-\frac{1}{\gamma\lambda(t)}\|\dot{u}(t)\|_M^2-\frac{1}{\gamma}\dot{h}(t),
	\end{align}
	which implies that
	\begin{align*}
		\alpha\lambda(t)\|\nabla f_\gamma(u(t))-\nabla f_\gamma(u^*)\| +\frac{d}{dt}\left(\frac{1}{\gamma}h+q\right)+\frac{1}{\gamma\lambda(t)}\|\dot{u}(t)\|_M^2\le 0.
	\end{align*}
	So the function $t\mapsto \frac{1}{\gamma}h+q$ is monotonically decreasing. Keeping in mind the proof of Proposition \ref{prop2}, and the fact that $\lambda$ has positive upper and lower bounds, we obtain that $\frac{1}{\gamma}h+q, h,q,u$ are bounded and $\dot{u}, (I-(I-\gamma M^{-1}\nabla g) \operatorname{prox}_{\gamma f}^M)u\in L^2([0,\infty);H)$. Also, $\lim\limits_{t\to \infty}\dot{u}(t)=0$. It follows from (\ref{1.1}) that $$\nabla f_\gamma(u(t))-\nabla f_\gamma(u^*)\in L^2([0,\infty);H).$$ From Lemma \ref{lm2.2}, we conclude $$\lim\limits_{t\to \infty}\nabla f_{\gamma}(u(t))= \nabla f_\gamma(u^*).$$ Therefore, statements (i), (ii) and (iv) are proved.
	
	(iii) First we show that every weak sequential cluster point of $u(\cdot)$ is in $\operatorname{Zer}((\nabla g)_{-\gamma}+\nabla f_\gamma)$. Let $u^*\in \operatorname{Zer}((\nabla g)_{-\gamma}+\nabla f_\gamma)$ and $t_n \to \infty$ (as $n\to \infty$) be such that $\{u(t_n)\}\rightharpoonup\bar{u}$. Since $(u(t_n), \nabla f_\gamma (u(t_n))\in \operatorname{gra}(\nabla f_\gamma)$, $\lim\limits_{n\to \infty}\nabla f_{\gamma}(u(t_n))= \nabla f_\gamma(u^*)$ and $\operatorname{gra}(\nabla f_\gamma)$ is sequentially closed in the weak-strong topology, we get $\nabla f_\gamma (\bar{u})= \nabla f_\gamma (u^*)$.
	
	Using the fact that $\operatorname{gra}(\nabla g)$ is sequentially closed in the weak-strong topology and letting $A=\partial f$, $B=\nabla g$ and $t=t_n$, we have $-\nabla f_\gamma (u^*)\in (\nabla g)_{-\gamma}(\bar{u})$. So, we get $-\nabla f_\gamma (\bar{u})\in (\nabla g)_{-\gamma}(\bar{u})$, hence $\bar{u}\in \operatorname{Zer}((\nabla g)_{-\gamma}+\nabla f_\gamma)$.
	
	Next, we show that $u(\cdot)$ has at most one weak sequential cluster point. It proves that the trajectory convergence weakly to a zero of $(\nabla g)_{-\gamma}+\nabla f_\gamma$.
	
	Let $u^*, v^*$ be two weak sequential cluster point of $u(\cdot)$. So, there exist sequences $\{t_n\}\to \infty $ and $\{t'_n\}\to \infty$ such that $\{u(t_n)\}\rightharpoonup u^*$ and $\{u(t'_n)\}\rightharpoonup v^*$. Since $u^*, v^*\in \underset{u\in H}{\operatorname{Argmin}}\{f(u)+g(u)\}\ne\emptyset$, we have $\lim\limits_{t\to \infty} \Phi(t, u^*)\in \mathbb{R}$ and $\lim\limits_{t\to \infty} \Phi(t, v^*)\in \mathbb{R}$, hence $\exists \lim\limits_{t\to \infty} \Phi(t, u^*)-\lim\limits_{t\to \infty} \Phi(t, v^*)\in \mathbb{R}$, where $$\Phi(t, u^*)= \frac{1}{2\gamma}\|u(t)-u^*\|^2+f(u(t))- f(u^*)- \langle \nabla f_\gamma(u^*), u(t)-u^* \rangle. $$ So, we get
	\begin{align}
		\exists \lim\limits_{t\to \infty} \left( \frac{1}{\gamma}\langle u(t), v^*-u^* \rangle+ \langle \nabla f_{\gamma}(u^*)- \nabla f_\gamma(v^*), u(t)\rangle \right)\in \mathbb{R}. \label{3.14}
	\end{align}
	If we express (\ref{3.14}) by the means of sequences $\{t_n\}$ and $\{t'_n\}$, we obtain
	\begin{align}
		&\frac{1}{\gamma}\langle u^*, v^*-u^* \rangle +\langle  \nabla f_\gamma(v^*)- \nabla f_\gamma(u^*), u^*\rangle\nonumber\\
		&= \frac{1}{\gamma}\langle v^*, v^*-u^* \rangle +\langle  \nabla f_\gamma(v^*)- \nabla f_\gamma(u^*), v^*\rangle,\nonumber
	\end{align}
	which implies that
	\begin{align}
		\frac{1}{\gamma}\|u^*- v^*\|+ \langle  \nabla f_\gamma(v^*)- \nabla f_\gamma(u^*), v^*-u^*\rangle=0,\nonumber
	\end{align}
	and by the monotonicity of $\nabla f_\gamma$ we conclude that $u^*= v^*$.
	
	(v) The proof follows by using $A=\partial f$ and $B=\nabla g$ in Theorem \ref{th1}(vii).
\end{proof}
\section{Numerical Examples}\label{sc3}
\begin{example} \label{ex2}\normalfont
	Let $H= \mathbb{R}$ be a Hilbert space endowed with Euclidean inner product. Let $A: H\to 2^H$ be a set-valued operator defined by
	\begin{eqnarray}{
			A(x)= \left\{
			\begin{array}{lc@{}c@{}r}
				0,~~ ~~~~~if~ x<0\\[6pt]
				[0,1],~~ if ~x=0\\[6pt]
				1,~~ ~~~~~if~ x>0.
			\end{array}\right.
		}\nonumber
	\end{eqnarray}
	Note that $A$ is maximally monotone operator. So, $A$ is $\rho$-cohypomonotone operator for $\rho\ge0$.
	\\ Let $B: \mathbb{R}\to \mathbb{R}$ be $1$-cocoercive operator defined by $$B(x)= \frac{x}{2}.$$ So $\beta= 1$.
	Let $M: \mathbb{R}\to \mathbb{R}$ be a strongly positive operator defined by
	$M(x)= 3x$. So, from (\ref{meq}), we have the dynamical system
	\begin{eqnarray}{
			\left\{
			\begin{array}{lc@{}c@{}r}
				\dot{u}(t)=
				{
					\lambda(t) \left\{
					\begin{array}{lc@{}c@{}r} \frac{11u(t)}{12}-u(t), ~~~~~~~t<0\\[6pt]
						-u(t),~~~~~~~~~~~~~~~ t\in [0,\frac{1}{6}]\\[6pt]
						\frac{11u(t)}{12}-\frac{1}{6}-u(t), ~~ t>\frac{1}{6}
					\end{array}\right.
				}\\
				u(0)= u_0.
			\end{array}\right.\nonumber
	}\end{eqnarray}
	Choose $\rho= \frac{1}{4}= \frac{1}{2}-\frac{1}{4}$. So, $\gamma= \frac{1}{2}, \alpha= \frac{1}{4}$. Let $\kappa =\frac{1}{2}$. Observe that all the assumptions of Theorem \ref{th1} are satisfied. Figure \ref{Figure:1} shows the convergence behaviour of the trajectories generated by dynamical system (\ref{meq}) for the Lebesgue measurable function $\lambda: [0,\infty)\to [0,\infty)$ defined by
	\begin{eqnarray}\label{4.1}
		{
			\lambda(t)=\left\{
			\begin{array}{lc@{}c@{}r}
				1,~~ if~ t\in [0, 50]\\[6pt]
				0,~~ otherwise.
			\end{array}\right.
	}\end{eqnarray}
\end{example}
\begin{example}\label{ex1}\normalfont
	Let $H=\mathbb{R}^2$ be a real Hilbert space endowed with Euclidean inner product and $N:\mathbb{R}^2\to \mathbb{R}^2$ be an operator defined by
	\[ N=
	\begin{bmatrix}
		-2 &     0   \\
		0 &    0
	\end{bmatrix}
	.\]\\
	Now, consider the multivalued operator $A:=N^{-1}$, so by Example \ref{re}, $A_{\rho}= (N+\rho I)^{-1}$ is maximally monotone operator for $\rho> 2$, and hence $A$ is $\rho(>2)$-cohypomonotone operator.
	Let $B: \mathbb{R}^2\to \mathbb{R}^2$ be $1$-cocercive operator defined by
	\begin{align}
		B= \begin{bmatrix}
			\frac{1}{2} &     0   \\
			0 &    \frac{1}{3}\nonumber
		\end{bmatrix}.
	\end{align} Here $\beta= 1$. Let $M: \mathbb{R}^2\to \mathbb{R}^2$ be a strongly positive operator defined by
	\begin{align}
		M=\begin{bmatrix}
			4 &     0   \\
			0 &    8
		\end{bmatrix}.\nonumber
	\end{align}
	So, from (\ref{meq}), we have the dynamical system
	\begin{eqnarray}{
			\left\{
			\begin{array}{lc@{}c@{}r}
				\dot{u}(t)=
				\lambda(t)\left[ \begin{bmatrix}
					\frac{1}{4} &     0   \\
					0 &    \frac{35}{48}
				\end{bmatrix}u(t)-u(t)\right]
				
				\\
				u(0)= u_0.
			\end{array}\right.\nonumber
	}\end{eqnarray}
	Choose $\rho= 3= 4-1$. So, $\gamma= 4, \alpha=1$. Let $\kappa =3$. Observe that all the assumptions of the Theorem \ref{th1} are satisfied. Figure \ref{Figure:2} shows the convergence behaviour of the trajectories generated by dynamical system (\ref{meq}) for the Lebesgue measurable function $\lambda: [0,\infty)\to [0,\infty)$ defined by (\ref{4.1}).
	
\end{example}

\begin{figure}[h]
	
	\begin{center}
		\includegraphics[width=8cm, height=7cm]{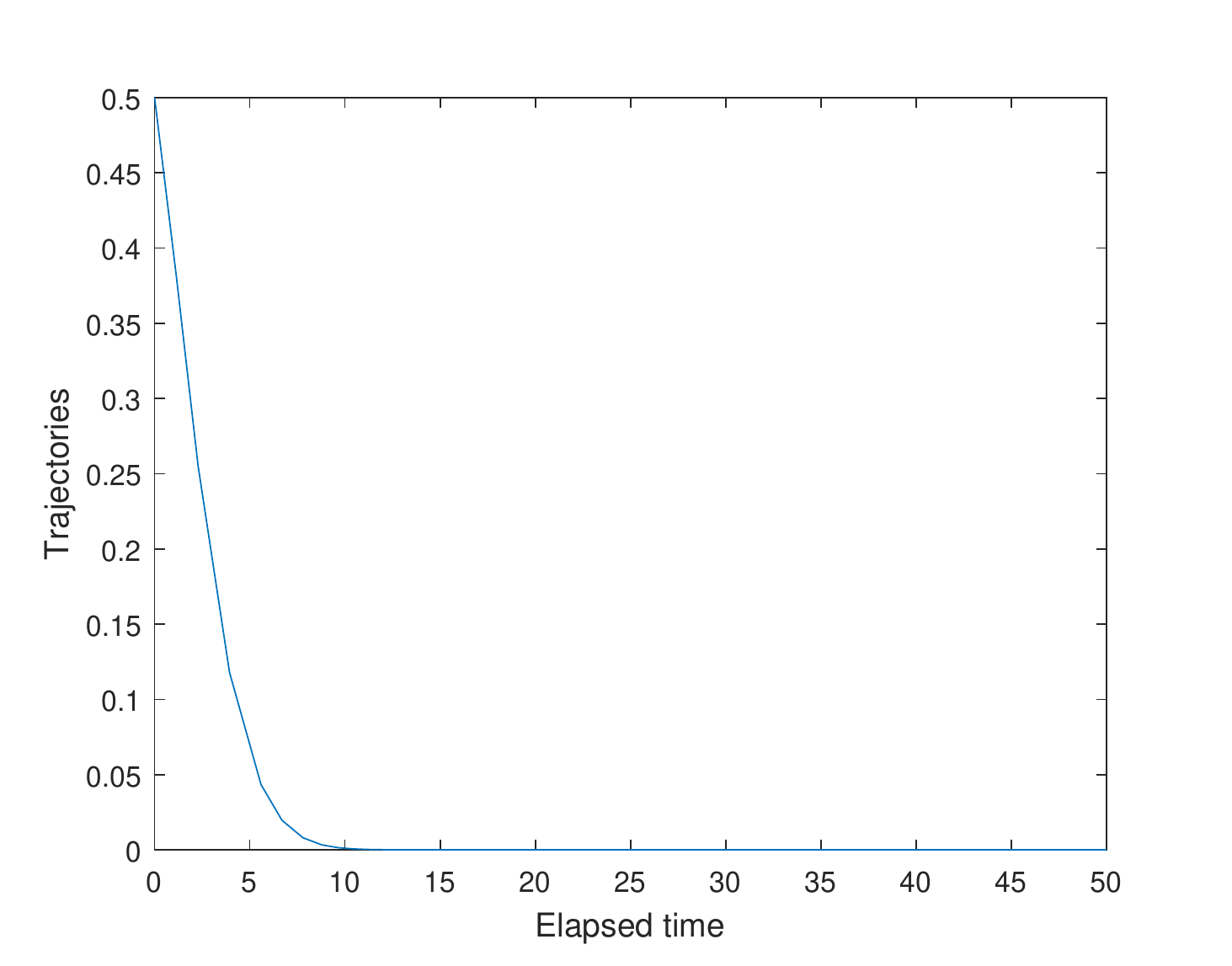}
	\end{center}
	\caption{Trajectories generated by the dynamical system of Example \ref{ex2} for $u_0= 0.5$.}
	\label{Figure:1}
\end{figure}
\begin{figure}[h]
	\begin{center}
		\includegraphics[width=8cm, height=7cm]{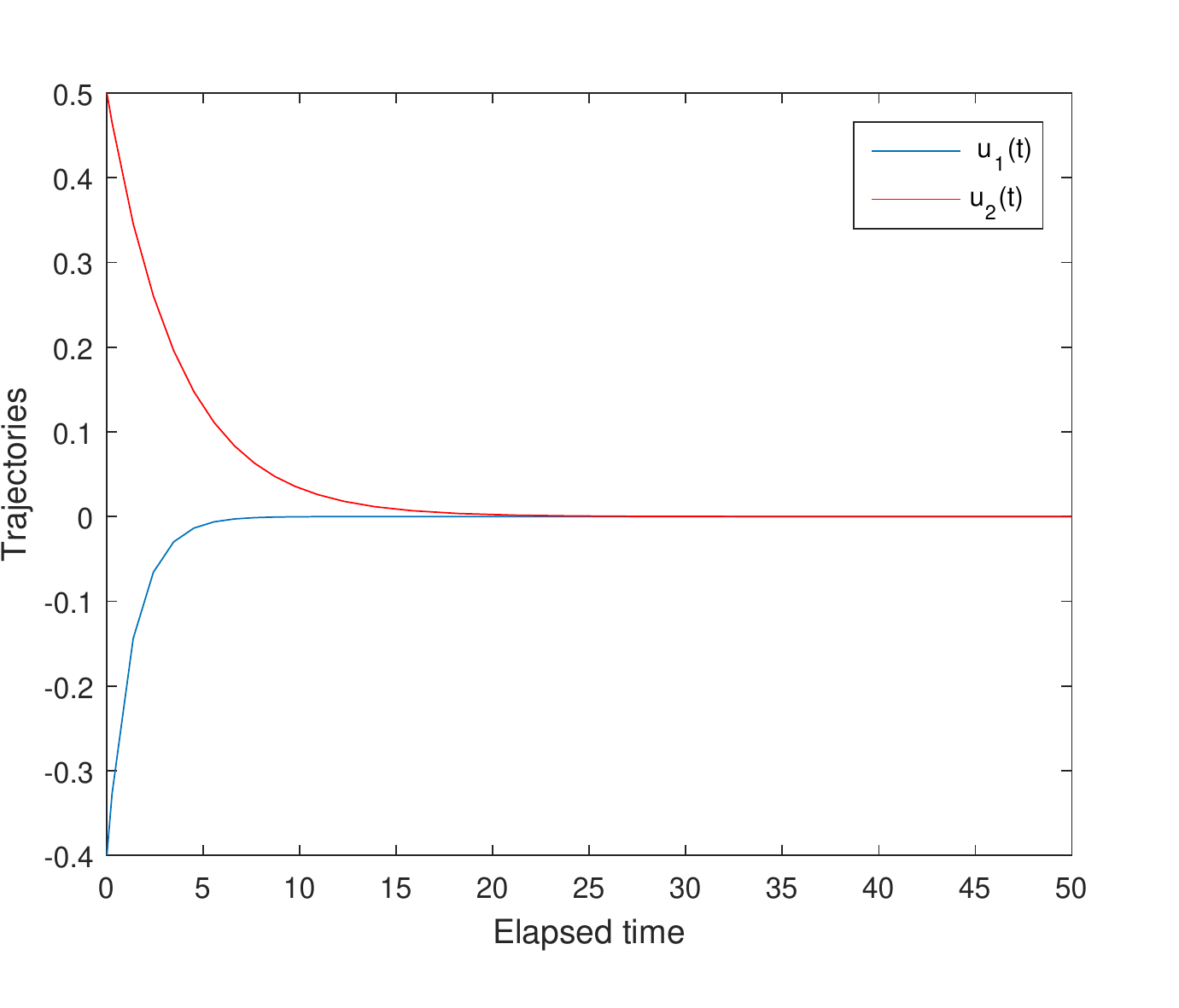}
	\end{center}
	\caption{Trajectories generated by the dynamical system of Example \ref{ex1} for $u_0= (-0.4, 0.5).$}
	\label{Figure:2}
\end{figure}
Note that discrete version of preconditioned backward-forward dynamical system is:
\begin{align}\label{n1}
	u_{n+1}= u_n+\lambda_n((I-\gamma M^{-1}B)(I-\gamma M^{-1}A_\gamma)u_n- u_n).
\end{align}
Taking a bounded sequence $\lambda_n= \frac{1}{n}$, one can observe by Example \ref{ex1} and Figure \ref{Figure:3} that sequence $\{\|u_n\|\}$ converges.
\begin{figure}[h]
	\begin{center}
		\includegraphics[width=8cm, height=7cm]{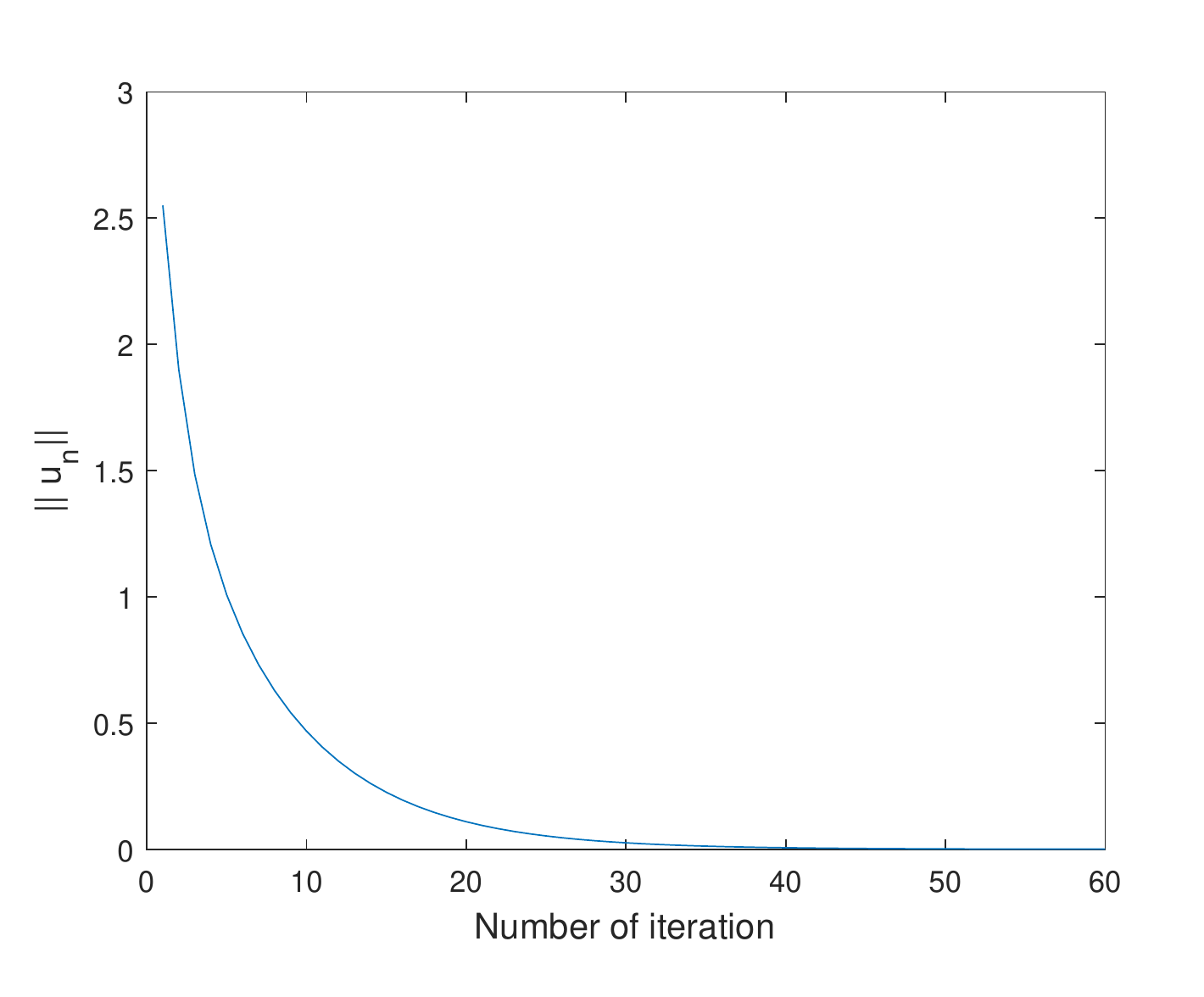}
	\end{center}
	\caption{Convergence of sequence $\{\|u_n\|\}$ given by (\ref{n1}) for Example \ref{ex1} and $u_0= (3, 2).$}
	\label{Figure:3}
\end{figure}
\section{Conclusions}\label{sc4}
In this paper, first-order variable metric backward-forward dynamical systems associated with monotone inclusion, and convex minimization problems have been studied. Existence, uniqueness, weak and strong convergence of the trajectories of dynamical systems (\ref{meq}), (\ref{meq2}), and (\ref{1}) have been studied. We have also established that an equilibrium point of the trajectory is globally exponentially stable and monotone attractor.
\section*{Acknowledgement}
The first author acknowledges the financial support from Ministry of Human Resource and Development (MHRD), New Delhi, India under Junior Research Fellow (JRF) scheme.

\end{document}